\documentclass[11pt]{amsart}%

\usepackage{amsmath,amsfonts,amsthm,amscd,amssymb,graphicx}
\usepackage[english]{babel}
\usepackage{dsfont}

\numberwithin{equation}{section}

\usepackage{hyperref}

\usepackage{color}

\newtheorem{theorem}{Theorem}[section]

\newtheorem{lemma}[theorem]{Lemma}
\newtheorem{lem}[theorem]{Lemma}
\newtheorem{proposition}[theorem]{Proposition}
\newtheorem{prop}[theorem]{Proposition}

\newtheorem{remark}[theorem]{Remark}

\renewcommand{\epsilon}{\varepsilon}

\def\pa{\partial}
\def\ep{\epsilon}

\def\na{\nabla}

\def\Re{\Re e}

\newcommand{\R}{\mathbb{R}}
\newcommand{\N}{\mathbb{N}}

\newcommand{\TT}{{\mathbb T}}

\def\e{\epsilon}
\def\eps{\varepsilon}

\def\cN{\mathcal{N}}

\def\beq{\begin{equation}}
\def\eeq{\end{equation}}

\def\app{\mathrm{app}}
\def\bl{\mathrm{bl}}
\def\inte{\mathrm{in}}

\newcommand{\Div}{\operatorname{div}}
\newcommand{\curl}{\operatorname{curl}}
\def\divz{\Div_z}
\def\curlz{\curl_z}

\begin{document}

\title[Vanishing viscosity and rugosity limit]{The vanishing viscosity limit for 2D Navier-Stokes \\
 in a rough domain}

\author[G\'erard-Varet, Lacave, Nguyen $\&$ Rousset]{David G\'erard-Varet, Christophe Lacave, Toan T. Nguyen and Fr\'ed\'eric Rousset}

\address[D. G\'erard-Varet]{Univ Paris Diderot, Sorbonne Paris Cit\'e, Institut de Math\'ematiques de Jussieu-Paris Rive Gauche,
UMR 7586, CNRS, Sorbonne Universit\'es, UPMC Univ Paris 06, F-75013, Paris, France.}
\email{david.gerard-varet@imj-prg.fr}

\address[C. Lacave]{Univ. Grenoble Alpes, IF, F-38000 Grenoble, France\\
CNRS, IF, F-38000 Grenoble, France.}
\email{Christophe.Lacave@univ-grenoble-alpes.fr}

\address[T. Nguyen]{Department of Mathematics, Pennsylvania State University, State College, PA 16802, USA.}
\email{nguyen@math.psu.edu}

 \address[F. Rousset]{Laboratoire de Math\'ematiques d'Orsay (UMR 8628), Universit\'e Paris-Sud et Institut Universitaire de France, 91405 Orsay Cedex, France}
 \email{frederic.rousset@math.u-psud.fr}

\begin{abstract} 
We study the high Reynolds number limit of a viscous fluid in the presence of a rough boundary. We consider the two-dimensional incompressible Navier-Stokes equations with Navier slip boundary condition, in a domain whose boundaries exhibit fast oscillations in the form $x_2 = \varepsilon^{1+\alpha} \eta(x_1/\varepsilon)$, $\alpha > 0$. Under suitable conditions on the oscillating parameter $\varepsilon$ and the viscosity $\nu$, we show that solutions of the Navier-Stokes system converge to solutions of the Euler system in the vanishing limit of both $\nu$ and $\eps$. The main issue is that the curvature of the boundary is unbounded as $\eps \rightarrow 0$, which precludes the use of standard methods to obtain the inviscid limit. Our approach is to first construct an accurate boundary layer approximation to the Euler solution in the rough domain, and then to derive stability estimates for this approximation under the Navier-Stokes evolution.
\end{abstract}

\maketitle

\section{Introduction}
Our concern in this paper is the inviscid limit of the Navier-Stokes equation in a domain with rough boundaries. We use a standard modelling of the roughness, through a small amplitude and small wavelength oscillation. Precisely, we consider the domain 
$$\Omega^\eps := \{ x = (x_1,x_2), \quad x_1 \in \TT, \: x_2 > \varepsilon^{1+\alpha} \eta(x_1/\varepsilon) \} $$
where $\eta = \eta(y_1)$ is a smooth and positive function of $y_1 \in \TT$. The boundary $\pa \Omega^\eps$ of the domain oscillates at wavelength $\eps$ and has typical amplitude $\eps^{1+\alpha}$, where $\alpha$ will be specified later. It is implicit here that $1/\eps$ is an integer, to be consistent with the periodicity of the domain in $x_1$. We consider in this rough domain a viscous fluid, governed by the incompressible Navier-Stokes equations, in the regime of high Reynolds number: ${\rm Re} = \frac1\nu \gg 1$. 
We assume that a slip boundary condition of Navier type holds at $\pa \Omega^\eps$. The system under consideration is therefore: 
\begin{equation} \label{NS}
\begin{aligned}
 u^{\nu,\varepsilon}_t + u^{\nu,\varepsilon}\cdot \nabla u^{\nu,\varepsilon} + \nabla p^{\nu,\varepsilon} - \nu \Delta u^{\nu,\varepsilon} & = f &\text{in } \Omega^\eps, \\
\Div u^{\nu,\varepsilon} & = 0 & \text{in } \Omega^\eps, \\
u^{\nu,\varepsilon} \cdot n^{\epsilon} = 0, \quad 2D(u^{\nu,\varepsilon})n^{\epsilon} \cdot \tau^{\epsilon}+ \lambda u^{\nu,\eps} \cdot \tau^\eps & = 0 & \text{ on } \pa \Omega^\eps.
\end{aligned}
\end{equation}
As usual, $u^{\nu,\eps} = u^{\nu,\eps}(t,x)$ and $p^{\nu,\eps} = p^{\nu,\eps}(t,x)$ are the velocity and pressure fields. The unit vectors $n^{\epsilon} = 
(-\eps^\alpha \eta', 1)/\langle \eps^\alpha \eta' \rangle$ and $ \tau^{\epsilon} = (1,\eps^\alpha \eta')/\langle \eps^\alpha \eta' \rangle$, with $\langle \eps^\alpha \eta' \rangle = \sqrt{1+\eps^{2 \alpha} |\eta'|^2}$, are normal and tangent to $\pa \Omega^\eps$. The first boundary condition expresses no penetration at the boundary. The second one, of a Navier type, expresses the shear stress: $ D(u) = \frac12 (\nabla u + (\nabla u)^T)$ is the deformation tensor and $\lambda$ is a scalar friction function of class $C^2$. To avoid a tedious discussion on the compatibility conditions that the initial data should satisfy, we consider that the flow is generated by a smooth forcing $f = f(t,x)$ which is zero in the past: namely, we assume that the source term $f$ in the momentum equation (\ref{NS}) satisfies 
$$f \in C^\infty(\R, H^\infty(\Omega^0)), \quad f\vert_{t < 0} = 0, $$ 
where $\Omega^0 = \TT \times \R_+$ is the {\em flat} domain. As $\inf \eta > 0$, one has $\Omega^\eps \Subset \Omega^0$ for $\eps$ small enough. 

Under such condition on $f$, there is a unique solution $u^{\nu,\eps} \in C^\infty(\R, H^\infty(\Omega^\eps))$ of \eqref{NS} satisfying 
\begin{equation} \label{uneg} 
u\vert_{t < 0} = 0. 
\end{equation} 
More precisely, as we consider the 2D Navier-Stokes equation, there is a unique global Leray solution. The smoothness of this solution is then obtained by a bootstrap argument, performing time differentiations of a sequence of smooth approximations, and taking the limit of this sequence. We refer to \cite{BoyerFabrie} in the case of classical Dirichlet conditions, and to \cite{Kelliher} in the case of Navier conditions. 

\medskip
{\em Our purpose is to understand the joint asymptotics limits $\nu \rightarrow 0, \eps \rightarrow 0$ of \eqref{NS}-\eqref{uneg}}. We want to find sufficient conditions under which the limiting behaviour is provided by the Euler system in the flat domain 
\begin{equation} \label{Euler}
\begin{aligned}
 u^{0}_t + u^{0}\cdot \nabla u^{0} + \nabla p^{0} & =f & \text{ in } \Omega^0, \\
\Div u^{0} & = 0 & \text{ in } \Omega^0, \\
u^0 \cdot n & =0 & \text{ on } \pa \Omega^0, \\
u^0\vert_{t<0} & = 0& \text{ in } \Omega^0. 
\end{aligned}
\end{equation}
There are various motivations for studying such a problem. First, due to the development of microfluidics, the nature of the interaction between a viscous fluid and a rough boundary has regained much interest. Recently, several experiments showed that rough hydrophobic surfaces may generate significant slip lengths at the boundary, resulting in drag decrease \cite{BocBar,YbBaCo,FeBaVi}. In this regard, considering a Navier boundary condition together with a rough boundary may be meaningful. Moreover, it is well-known that the development of instabilities at a high Reynolds number is often triggered by wall roughness, see for instance \cite{CabFlo,SzuFlo}. Hence, studying the combined effect of $\eps$ and $\nu$ in Navier-Stokes has great physical relevance. 
 
 Another main motivation is a better mathematical understanding of the vanishing viscosity limit in the presence of boundaries. 
 In the case without boundaries, the convergence of smooth solutions of Navier-Stokes to smooth solutions of Euler as $\nu \rightarrow 0$ has been known for long (see \cite{Swann71,Kato72,Kato86,Masmoudi07}). However, when one considers smooth boundaries with the usual no-slip condition, this problem is essentially open. As the Euler flow does not satisfy the no-slip conditions, the convergence can not hold in strong topology (say $H^1$). Hence, one can not bound the velocity gradients uniformly in $\nu$ near the boundary: this is a boundary layer phenomenon, which may preclude even the $L^2$ convergence of Navier-Stokes to Euler. Following the classical approach of Prandtl \cite{Prandtl04}, the starting idea is that the Navier-Stokes solution should admit an expansion of the form 
 $$ u^\nu(t,x) \sim u^0(t,x) + U(t,x_1,x_2/\sqrt{\nu}) + v^\nu(t,x) = u^{\app,\nu}(t,x) + v^\nu(t,x)$$
near the boundary $x_2=0$, where $u^0$ is the Euler solution, $U$ describes a boundary layer corrector with typical scale $\sqrt{\nu}$ transversally to the boundary, and $v^\nu$ is a small perturbation. However, as the boundary layer approximation satisfies $\na u^{\app,\nu} \sim \frac{1}{\sqrt{\nu}}$, it may stretch a lot the perturbation, and yields instability in very short time scales. See \cite{Grenier00CPAM} for a discussion of this phenomenon, or the recent works \cite{GGN,GMM}. 
 
 In the case of Navier conditions at smooth boundaries (system (\ref{NS}) with $\eps=1$), these instabilities are filtered out: there is still a boundary layer, but of a smaller amplitude: the formal asymptotics is rather
 $$ u^\nu(t,x) \sim u^0(t,x) + \sqrt{\nu} U(t,x_1,x_2/\sqrt{\nu}) + v^\nu(t,x) = u^{\app,\nu}(t,x) + v^\nu(t,x),$$
near the boundary $x_2=0$; see \cite{IftimieSueur} for details. One can show the convergence of the Navier-Stokes solutions to the Euler solutions in several ways: either through direct energy estimates of $u^\nu - u^0$, or by considering the equation for the vorticity $\omega = \curl u$. Indeed, the Navier condition at a smooth boundary $\pa \Omega$ can be written: $\omega\vert_{\pa \Omega} = (2\kappa+\lambda) v \cdot \tau\vert_{\pa \Omega}$, with $\kappa$ the curvature of the boundary and $v \cdot \tau$ the tangential velocity. This allows to control $\omega$ through maximum principle arguments, and to show convergence to Euler by strong compactness arguments: see \cite{MikRob} for more. This strong compactness approach has also been carried out in three dimensions \cite{MasRou1} and for other type of boundary conditions
involving normal derivatives of the velocity, in particular for free surface fluids \cite{MasRou2}, by propagating
higher conormal regularity. 

Interestingly, the methods just mentioned fail for the case of rough boundaries when $\eps \rightarrow 0$ and $\alpha < 1$: for instance, the curvature of the boundary is now 
$$\kappa = \frac{\eps^{\alpha-1} \eta''}{\langle \eps^\alpha \eta' \rangle^3}. $$ 
It is therefore unbounded as $\eps \rightarrow 0$, and so is the vorticity $\omega^{\nu,\eps}$. Another way to emphasize this is to consider the Euler solution in the rough domain $\Omega^\eps$, that is $u^{0,\eps}$. As will be shown later, it admits an expansion in the form 
$$ u^{0,\eps}(t,x) \sim u^0(t,x) + \eps^\alpha U(t,x_1,x/\eps) $$
where $U$ describes now an inviscid boundary layer corrector generated by the roughness. When $\alpha$ is very close to zero, or zero, we see that we get closer in spirit to a Prandtl expansion, with possible instabilities. This is a strong mathematical motivation for our study. 

A last indication of the difficulty of this convergence problem is that for $\alpha = 0$, the limits $\nu \rightarrow 0$ and $\eps \rightarrow 0$ do not commute. First, if $\nu \rightarrow 0$ at a fixed $\eps$, one recovers the Euler equation in the rough domain $\Omega^\eps$, {\it cf} the previous discussion. Then, if $\eps \rightarrow 0$, one recovers the Euler equation in the flat domain: this is a consequence of general continuity results established by the first and second authors in \cite{GVLacave13,GVLacave15}: the Euler solution is continuous with respect to its domain, in the sense that $\gamma$-convergence of the domain implies convergence of the solution in $L^2$ topology. Let us mention that $\alpha > - 1$ is enough in this step. 
On the other hand, if one considers first the limit $\eps \rightarrow 0$ at a fixed $\nu$, if $\eta \neq 0$, the limit system is the Navier-Stokes equation in $\Omega^0$, {\em with a no-slip condition at $\pa \Omega^0$}. This surprising change in the boundary condition is due to strong dissipation near the rough boundary, and shown in \cite{CFS}, see also \cite{BFNW,BonBuc,DalibardGV}. Then, if one tries to send $\nu$ to zero, one faces the usual problem associated with Dirichlet conditions. This emphasizes that considering a joint asymptotics in $\eps, \nu$ is relevant. 

\medskip
Due to all above observations, our main goal is to exhibit some asymptotic regimes in $(\eps,\nu)$ where we can obtain convergence to the Euler flows. More precisely, we will prove convergence as soon as $\alpha > 0$ and if $\nu$ is sufficiently small compared to $\epsilon$. We will assume that 
\begin{equation} \label{constraint_alpha}
 \alpha=\frac{1}{N_0}, \quad \text{ with $N_0$ an arbitrarily large integer number.}
\end{equation}
 In particular, $\alpha < 1$. Note that for $\alpha \ge 1$, the curvature of the boundary is bounded, and the convergence can be deduced from the compactness argument of \cite{MikRob}, and no new argument is needed. For $\alpha < 1$, we proceed in two steps to prove convergence. 
 
 \medskip
First, we construct a good approximate solution of the Euler equations in $\Omega^\eps$, involving boundary layers due to the roughness. 
\begin{proposition}\label{prop-Euler} Let $f \in C^\infty(\R, H^\infty(\Omega^0))$ with $f\vert_{t <0} = 0$, and $u^0 \in C^\infty(\R, H^\infty(\Omega^0))$ the solution of the Euler equations \eqref{Euler}. For any $M > 0$, there exist approximate velocity and pressure 
$$
u^\app = u^0 + u^\app_\inte + u^\app_\bl, \qquad 
p^\app = p^0 + p^\app_\inte + p^\app_\bl 
$$
satisfying 
\begin{subequations}
\begin{align}
 \partial_t u^\app + ( u^\app \cdot \nabla ) u^\app + \nabla p^\app & = f + R_\inte^\app + R_\bl^\app & \text{in } \Omega^\varepsilon, \label{Eav-eqs} \\
\Div u^\app &= 0 & \text{in } \Omega^\varepsilon, \label{Eadiv-eqs} \\
u^\app \cdot n^\e &= 0 & \text{on } \pa \Omega^\epsilon \label{Eabdry-eqs}, \\
u^\app\vert_{t < 0} & = 0 &\text{in } \Omega^\varepsilon,
\end{align}
\end{subequations}
and the following estimates, for any $T_0 > 0$, $\gamma \in (0,1)$, $\beta \in \mathbb{Z}^2$: 
\begin{equation} \label{all_bounds}
\left\{
\begin{aligned} 
& \sup_{t\in [0,T_0]} \| \pa^\beta u^\app_\inte \|_{L^\infty(\Omega^\eps)} + \| \pa^\beta u^\app_\inte \|_{L^2(\Omega^\eps)} \lesssim \epsilon^{\alpha+1} ,\quad 
\sup_{t\in [0,T_0]} \| e^{\gamma x_2/\eps} \eps^{|\beta|} \partial^\beta u_\bl^\app \|_{L^\infty(\Omega^\eps)} \lesssim \epsilon^{\alpha}, \\
& \sup_{t\in [0,T_0]} \| e^{\gamma x_2/\eps} \eps^{|\beta|} \partial^\beta u_\bl^\app \|_{L^2(\Omega^\eps)} \lesssim \epsilon^{\alpha+\frac12}, \quad 
 \sup_{t\in [0,T_0]} \| e^{\gamma x_2/\eps} \eps^{|\beta|} \partial^\beta \curl u_\bl^\app\|_{L^\infty(\Omega^\eps)} \lesssim \epsilon^{M}, \\
& \sup_{t\in [0,T_0]} \| \partial^\beta R_\inte^\app\|_{L^\infty(\Omega^\varepsilon)}+ \| \partial^\beta R_\inte^\app\|_{L^2(\Omega^\varepsilon)} \lesssim \epsilon^M, \\
& \sup_{t\in [0,T_0]} \| e^{\gamma x_2/\eps} \eps^{|\beta|} \partial^\beta R_\bl^\app\|_{L^\infty(\Omega^\varepsilon)} \lesssim \epsilon^{\alpha+1}, \quad 
 \sup_{t\in [0,T_0]} \| e^{\gamma x_2/\eps} \eps^{|\beta|} \partial^\beta R_\bl^\app\|_{L^2(\Omega^\varepsilon)} \lesssim \epsilon^{\alpha+\frac{3}{2}},
 \end{aligned}
 \right.
 \end{equation}
 in which $\pa$ denotes derivatives with respect to $x$. 
\end{proposition}
In this proposition and throughout the paper, the notation $g\lesssim h$ is used for bounds of the form $g \le C h$, for constants $C$ that are independent of $\varepsilon$ and $\nu$. 
\begin{remark}\label{rem-dtau}
As will be clear from the proof, the bounds \eqref{all_bounds} also apply to the time derivatives of all quantities. 
\end{remark}
\begin{remark}
The Euler approximation in the proposition is consistent with the results in \cite{GVLacave13,GVLacave15}, which show convergence of 
the Euler solution $u^{0,\eps}$ in $\Omega^\eps$ to the Euler solution $u^0$ in $\Omega^0$, in $L^2$ topology, for any $\alpha \ge 0$. We take advantage here of the special structure of the rough boundary to get a more accurate description of $u^{0,\eps}$: $u^\app_\inte$ is a macroscopic correction to $u^0$ of amplitude $\eps^{\alpha+1}$, while $u^{\app}_{\bl}$ describes a boundary layer of amplitude $\eps^\alpha$, typical scale $\eps$, and almost curl-free. 
\end{remark}

The second step of our approach, which is the central one, is to derive stability estimates for the previous approximation under the Navier-Stokes evolution. As this approximation has unbounded gradient as $\eps \rightarrow 0$, this stability does not follow from a standard energy estimate and a Gronwall lemma on $v = u^{\nu,\eps} - u^\app$. Our strategy is the following. In the $L^2$ estimate for the velocity, we write the bad stretching term as 
$$ \int_{\Omega^\eps} (v \cdot \na u^\app) \cdot v = \int_{\Omega^\eps} v \cdot (u^\app)^\perp \, \curl v $$
which is bounded by the product of the $L^2$ norms of the velocity and the vorticity. This requires in turn an $L^2$ estimate for the vorticity. This is not direct, as $\omega$ does not satisfy a good boundary condition. We overcome this problem through careful weighted estimates, with a weight that is of boundary layer type and vanishes at the boundary. Note that in this estimate, the smallness of $\curl u^\app_\bl$ is crucial. Combining the $L^2$ estimates for velocity and vorticity yields some stability, {\em under a smallness assumption on }  $\na v$ in $L^\infty$. The last part of our analysis is devoted to showing that such assumption holds. It mixes maximum principle estimates, time derivative estimates and inequalities of harmonic analysis in the rough domain. In deriving such inequalities, one must be very careful about the oscillations of the boundary. 

\medskip
Eventually, we prove: 
\begin{theorem}\label{theo-main} Let $f$, $u^0$, $(u^\app,p^\app)$ as in Proposition~\ref{prop-Euler}. Let $T_0 > 0$, 
$N_1 \in \mathbb{N}^*$ arbitrarily large. Then there exists $\eps_0$ such that for all $\eps \le \eps_0$, for all $\nu$ so that 
\begin{equation}\label{cond-nu-e}\varepsilon^{N_{1}} \lesssim \nu \lesssim \e^7,\end{equation}
and all $\lambda$ such that
\begin{equation*}
| \lambda |_{C^2} \lesssim \e^{-1+\alpha},
\end{equation*}
 the unique solution $u^{\nu,\eps} \in C^\infty(\R, H^\infty(\Omega^\eps))$ of the Navier-Stokes equations \eqref{NS} satisfies
$$ \sup_{0\le t \le T_0}\Big( \varepsilon^{-1/2} \| (u^{\nu,\eps} - u^\app)(t) \|_{L^2(\Omega^\varepsilon)} + \| (u^{\nu,\eps} - u^\app)(t)\|_{L^\infty(\Omega^\varepsilon)} + \e \| \curl (u^{\nu,\eps} - u^\app)(t)\|_{L^\infty(\Omega^\varepsilon)}\Big)\lesssim \e^\alpha.$$
 \end{theorem}

The constraint \eqref{cond-nu-e} is explained in Remark~\ref{rem-enu}. As a corollary of Proposition~\ref{prop-Euler} and Theorem~\ref{theo-main}, we obtain the following vanishing viscosity and rugosity limit:
$$ \sup_{0\le t \le T_0}\Big( \| u^{\nu,\e} - u^0 \|_{L^2(\Omega^\varepsilon)} + \| u^{\nu,\e} - u^0\|_{L^\infty(\Omega^\varepsilon)} \Big) \to 0,$$
in the limit $(\varepsilon, \nu) \to 0$, provided $\varepsilon^{N_{1}} \lesssim \nu \lesssim \e^7$ and $| \lambda |_{C^2} \lesssim \e^{-1+\alpha}$.

\medskip
Let us stress again that the novelty and difficulty of this inviscid limit result lie in the consideration of the {\em joint asymptotics} $(\nu, \eps) \rightarrow 0$. In this way, it is very different from both 
\begin{itemize}
\item the asymptotic results for Navier-Stokes equations in rough domains ($\nu = 1, \eps \rightarrow 0$), like in \cite{JagerMikelic,BFNW,MikNec}. 
\item the inviscid limit result for the Navier-Stokes equations with slip law in a smooth domain ($\nu \rightarrow 0$, $\eps = 1$), like in \cite{IftimieSueur}. 
\end{itemize}
In those latter cases, the asymptotic description of the Navier-Stokes solution also involves the construction of refined approximations, including boundary layer correctors. But to establish the $L^2$ stability of these approximations is quite easy, as the gradient of the approximate velocity is bounded in $L^\infty$. This is in sharp contrast with the present setting, in which the gradient diverges with the roughness parameter $\ep$, making the stability analysis the core of our paper. 

On the technical point of view, our approach can be compared to the recent works \cite{MasRou1, MasRou2}, where uniform $H^s$ type estimates are obtained for Navier-Stokes solutions in smooth domains, endowed with a Navier boundary condition or with a free surface. Indeed, our proof of stability uses vorticity and high order estimates which borrow a little to the methodology of \cite{MasRou1, MasRou2}. Still the difficulties related to the roughness are specific to our work. Indeed, in \cite{MasRou1, MasRou2} the estimates close again because the velocity has bounded gradient
in $L^\infty$ which is not true here.

Furthermore, our vorticity estimates will be based on rather elementary weighted estimates, very distinct from the elaborate semiclassical arguments used in \cite{MasRou1}. 

Eventually, although it is further from our concern, let us also mention the recent article \cite{LacaveMazzucato}: the authors consider a Navier-Stokes flow with no-slip in a porous medium, and prove convergence to the Euler flow in $L^\infty(0,T;L^2)$ in the simultaneous limit of vanishing porosity and viscosity. To avoid Prandtl instabilities (\cite{Grenier00CPAM,GGN, DGV,DGVN,GN}), a suitable assumption relates the viscosity and the porosity parameter.

%We end the introduction by mentioning the recent article \cite{LacaveMazzucato}, where a similar problem was studied: the authors considered the Navier-Stokes equations with the no-slip boundary condition in a porous medium, and proved the convergence to the Euler flows in $L^\infty(0,T;L^2)$ in the simultaneous limit of vanishing porosity and viscosity. To avoid Prandt instabilities (\cite{Grenier00CPAM,GGN, DGV,DGVN,GN}), a suitable condition was stated between the viscosity and the porosity parameter.

\section{Euler approximation}\label{sec-Euler}
In this section, we construct an approximate solution of the Euler equations in 
$$\Omega^\epsilon = \{ x_1 \in \TT, \: x_2 > \eps^{1+\alpha} \eta(x_1/\eps)\}, \text{ with }\frac1\alpha = N_{0}\in \N^*,$$
and give a proof of Proposition~\ref{prop-Euler}. 
The system reads 
\begin{equation} \label{Euler_eps}
\left\{
\begin{aligned}
\pa_t u + u \cdot \na u + \na p = f &\quad \mbox{in} \: \Omega^\eps, \\
{\rm div} \: u = 0 &\quad \mbox{in} \: \Omega^\eps, \\
u \cdot n^\eps = 0 &\quad \mbox{on} \: \pa \Omega^\eps, \\
u\vert_{t < 0} = 0&\quad \mbox{in} \: \Omega^\varepsilon. 
\end{aligned}
\right. 
\end{equation}
Here $n^\eps$ is the inward unit normal vector: 
$$ n^\eps(x_1) = \frac{1}{\langle \eps^\alpha \eta'(x_1/\eps)\rangle } (-\eps^\alpha \eta'(x_1/\eps), 1)$$ 
and $\langle \xi \rangle = \sqrt{1 + |\xi|^2}$ is the usual japanese bracket. The goal is to prove Proposition~\ref{prop-Euler}. Let $(u^0, p^0)$ the smooth solution of \eqref{Euler}.

\subsection{Boundary layer variables}
The pair $(u^0, p^0)$ is of course still a solution of the momentum equation in $\Omega^\eps \Subset \Omega^0$, but it does not satisfy the non-penetration condition (\ref{Euler_eps}c). Precisely, on $x_2 = \e^{1+\alpha} \eta(\frac{x_1}{\e})$, we compute

\begin{equation}\label{exp-Bv}
\begin{aligned}
u^0 \cdot n^\eps &= \frac{u_2^0 (t,x_1, \e^{1+\alpha} \eta )- \e^\alpha \eta' u_1^0 (t,x_1, \e^{1+\alpha} \eta ) }{\langle \eps^\alpha \eta'\rangle }
\\&= \frac{- \e^\alpha \eta' u_1^0(t, x_1, 0) + \e^{1+\alpha} \eta \partial_{x_2}u_2^0(t, x_1,0) - \e^{1+2\alpha} \eta \eta' \partial_{x_2}u_1^0 (t, x_1, 0) +\cdots}{\langle \eps^\alpha \eta'\rangle }
\\&= \sum_{k=1}^N \e^{\alpha k} B_k[u^0](t,x_1,\frac{x_1}{\e}) + E_0(t,x_1,\frac{x_1}{\e}), 
\end{aligned}\end{equation}
in which $E_0$ denotes the remainder term of order $\e^{\alpha (N+1)}$ or smaller. It is understood here that we do not expand the expression $\langle \eps^\alpha \eta' \rangle $ in powers of $\eps^\alpha$. In this way, there will be no term of order $\varepsilon^{\alpha k}$ in $u^\app$ for $1<k\leq N_{0}$: see \eqref{psivanishes}. It will also simplify the calculation of some boundary integrals, because $\int_{\partial\Omega^\varepsilon} \frac{f(x)}{\langle \eps^\alpha \eta' \rangle } d\sigma(x) = \int_{\mathbb{T}} f( x_{1}, \eps^{\alpha+1}\eta(x_{1}/\eps)) dx_{1}$. With this convention, 
$$B_1[u^0](t,x_1,\frac{x_1}{\e}) = -\frac{\eta'(\frac{x_1}{\e})}{\langle \eps^\alpha \eta' (\frac{x_1}{\e})\rangle } u_1^0(t,x_1,0), \quad \text{ while $B_k[u^0] = 0$ for all $1 < k \le N_0$.}$$ 
The dependence of operators $B_k$ with respect to $\eps$ is omitted in the notation. 

\medskip
To correct the error term at the boundary, we must add boundary layer terms. For $x \in \Omega^\e$, we introduce the boundary layer variable $z = \frac{x}{\eps}$, that belongs to $\Omega_\bl := \{ z_1 \in \TT, \: z_2 > \eps^\alpha \eta(z_1) \}$. The rescaled domain $\Omega_\bl$ still depends on $\eps$, but not singularly, so that again, we omit it in the notation. We look for an approximate solution in the form:
\begin{equation*}
u^\app(t,x) \approx u^0(t,x) + u^\app_\inte(t,x) + v^\app_\bl \left(t,x_1,\frac{x}{\eps}\right),
\end{equation*}
with boundary layer part 
$$ v^\app_\bl \left(t,x_1,z\right) = \sum_{k=1}^N \e^{\alpha k} v^k_\bl(t,x_1,z) $$
and interior part 
$$ u^\app_\inte (t,x): = \sum_{k=1}^N \e^{\alpha k} u^k(t,x).$$
The boundary layer profiles $v_\bl^k(t,x_1, z)$ will be defined for $x_1 \in \TT$, $z \in \Omega_\bl$, they will be periodic in $z_1$ (that is $z_1 \in \TT$) and will decay rapidly in $z_2$. The interior profiles $u_\inte^k(t,x)$ will be defined for $x \in \Omega^0$. They will arise in the construction process in order to relax compatibility conditions that the boundary layer profiles must satisfy. 

\subsection{Construction of boundary layer profiles}

We would like the boundary layer expansion 
\begin{equation} \label{uappbl}
 u^\app_\bl(t,x) = v^\app_\bl \left(t,x_1,\frac{x}{\eps}\right)
 \end{equation}
to satisfy the boundary condition 
$$ u^\app_\bl \cdot n^\eps \approx - (u^0 + \sum_{k=1}^N \eps^{\alpha k} u^k) \cdot n^\eps \quad \textrm{at} \: \pa \Omega^\eps,$$
as well as 
$$ {\rm div} \, u^\app_\bl \approx 0, \quad \curl u^\app_{\bl}= \partial_{x_{1}}u^\app_{\bl,2}-\partial_{x_{2}}u^\app_{\bl,1} \approx 0. $$
The last curl-free condition is a mathematical requirement, that will be essential to our stability analysis. It does not follow from the Euler dynamics, but as we shall see, it will not create a too large error term in the momentum equation. 
Using expansions \eqref{exp-Bv} and \eqref{uappbl}, we find 
\begin{equation} \label{BCbis}
\begin{aligned}
v^k_\bl(t,x_1,z) \cdot n(z_1) = & - (B_k[u^0] + B_{k-1}[u^1] + \dots + B_1[u^{k-1}])(t,x_1,z_1) \\
& - \: \frac{u^k_2(t,x_1,0) }{\langle \eps^\alpha \eta'\rangle }, \quad z \in \pa \Omega_\bl, 
\end{aligned}
\end{equation}
 as well as 
 $$ \divz v^k_\bl = - \pa_{x_1} v_{\bl,1}^{k - N_0}, \quad \curlz v^k_\bl = - \pa_{x_1} v_{\bl,2}^{k - N_0} \quad \textrm{in} \: \Omega_{bl}, $$
 with the convention $v^k_\bl = 0$ whenever $k \le 0$. 

\medskip
To satisfy the last two equations, we express the boundary layer profiles in terms of a potential and a stream function, 
\begin{equation} \label{stream_potential}
v_\bl ^k = \nabla_z \psi_\bl^k(t,x_1,z) + \nabla_z^\perp \phi_\bl^k (t,x_1,z)
\end{equation}
and impose
$$ \Delta_z \psi_\bl^k = - \pa_{x_1} v_{\bl,1}^{k - N_0}, \quad \Delta_z \phi_\bl^k = - \pa_{x_1} v_{\bl,2}^{k - N_0}. $$
As regards the boundary condition \eqref{BCbis}, it appears natural to prescribe 
$$ n(z_1) \cdot \nabla_z \psi_\bl^k(t,x_1,z) = - \sum_{j=0}^{k-1}B_{k-j}[u^j](t,x_1,z_1) - \frac{u^k_2(t,x_1,0) }{\langle \eps^\alpha \eta'\rangle }, \quad z \in \pa \Omega_\bl $$
with $n(z_1) = \frac{1}{\langle \eps^\alpha \eta'(z_1)\rangle } (-\eps^\alpha \eta'(z_1),1)$, and 
$$ \phi^k_\bl(t,x_1,z) = 0, \quad z \in \pa \Omega_\bl $$ 
(so that the normal component of $\nabla_z^\perp \phi_\bl^k$ vanishes at $\pa \Omega_{bl}$). However, a slight subtlety comes from the Laplace equation on $\psi^k_\bl$. Indeed, the source term and the boundary data must satisfy some compatibility condition, so as to ensure the existence of a solution whose gradient decays in variable $z_2$. To guess the right compatibility condition is easy: by integrating the Laplace equation on 
$\Omega_\bl$, using Stokes formula and the fact that $\int_{\pa \Omega_\bl} 1/\langle \eps^\alpha \eta'(z_{1})\rangle d\sigma(z) =\int_{0}^1 1\,ds=1$. We find that necessarily
\begin{equation} \label{BCint}
u^k_2\vert_{x_2 = 0} = h^k
\end{equation}
where 
\begin{equation} \label{hk}
h^k = h^k(t,x_1) = -\int_{\Omega_\bl} \pa_{x_1} v_{\bl,1}^{k - N_0}(t,x_1,z) dz - \int_{\pa \Omega_\bl} \sum_{j=0}^{k-1}B_{k-j}[u^j](t,x_1,z_1) .
\end{equation}
The appropriate boundary layer system on $\psi_\bl^k$ then becomes 
\begin{equation} \label{psikbl}
\left\{
\begin{aligned}
\Delta_z \psi_\bl^k & = - \pa_{x_1} v_{\bl,1}^{k - N_0}, \quad z \in \Omega_\bl, \\
n \cdot \na_z \psi_\bl^k & = - \sum_{j=0}^{k-1}B_{k-j}[u^j] -\frac{h^k}{ \langle \eps^\alpha \eta'(z_{1})\rangle }, \quad z \in \pa \Omega_\bl. 
\end{aligned}
\right.
\end{equation}
Concerning the stream function $\phi^k_\bl$, no compatibility condition is {\it a priori} needed. It satisfies 
\begin{equation} \label{phikbl}
\left\{
\begin{aligned}
\Delta_z \phi_\bl^k & = - \pa_{x_1} v_{\bl,2}^{k - N_0}, \quad z \in \Omega_\bl, \\
 \phi_\bl^k & = 0, \quad z \in \pa \Omega_\bl. 
\end{aligned}
\right. 
\end{equation} 
One must then address the solvability of the family of systems \eqref{psikbl}-\eqref{phikbl}, indexed by $k \in \N^*$.
We remind that defining $h^k$ through \eqref{hk} is necessary for the existence of a decaying solution to \eqref{psikbl}. This turns out to be a sufficient condition to solve all boundary layer systems, as follows from 
\begin{prop}
Under definition \eqref{hk}, the family of systems \eqref{psikbl}-\eqref{phikbl} has a unique family of smooth solutions $(\psi^k_\bl,\phi^k_\bl)$ (indexed by $k \in \N^*$) such that, for all $T_0 > 0$, for all $\tilde \gamma \in (0,1)$, for all $a,b,s \in \N$, there exists $C$ such that 
$$\sup_{t \in [0,T_0],x_1} \left( \| e^{\tilde \gamma z_2} \pa^a_t \pa^b_{x_1} \psi^k_\bl(t,x_1,\cdot) \|_{H^s(\Omega_\bl)} + \| e^{\tilde \gamma z_2}  \pa^a_t \pa^b_{x_1} \na_z \phi^k_\bl(t,x_1,\cdot) \|_{H^{s}(\Omega_\bl)} \right) \le C. $$
\end{prop}
\begin{proof}
The proposition is proved inductively on $k$. We first explain the case $a = b = 0$. We consider $t$ and $x_1$ as fixed parameters in these PDEs in variable $z$, and omit them temporarily from the notations. 

 The key is to show by induction on $k$ that \eqref{psikbl} and \eqref{phikbl} have smooth solutions 
$\psi^k_\bl$ and $\phi^k_\bl$ with the following property: for $z_2 > 1$, their Fourier series expansions in $z_1 \in \mathbb{T}$ are of the form 
\begin{equation} \label{Fourier_psi_phi}
\begin{aligned}
& \psi^k_\bl = \sum_{j \in \mathbb{Z}^*} e^{i j z_1} e^{-|j| z_2}P_j^k(z_2), \\
& \phi^k_\bl = Q^k_0 + \sum_{j \in \mathbb{Z}^*} e^{i j z_1} e^{-|j| z_2} Q^k_j(z_2) 
\end{aligned} 
\end{equation}
where $P^k_j$ and $Q^k_j$ are polynomials in $z_2$ for $j \in \mathbb{Z}^*$, while $Q^k_0$ is a constant. The proposition follows easily from such statement. 

In the case of \eqref{psikbl} and $\psi^k_\bl$, in which the compatibility condition is involved, this statement follows easily from \cite[Lemma 2.2]{Chupin}. In the case of \eqref{phikbl} and $\phi^k_\bl$, for which Dirichlet conditions hold at the boundary, the statement is even simpler to prove. In both cases, the existence and uniqueness of a weak solution is proved thanks to a Lax-Milgram lemma. The smoothness of the solution is deduced from the classical elliptic regularity. Eventually, behavior \eqref{Fourier_psi_phi} follows from the Fourier transform of the Laplace equations, which leads inductively to the ODE's 
$$ (\pa^2_{z_2} - j^2) \widehat{\psi^k_\bl}(j,z_2) = F_j(z_2), (\pa^2_{z_2} - j^2) \widehat{\phi^k_\bl}(j,z_2) = G_j(z_2) $$
with sources $F_j$ and $G_j$ being products of $\exp(-|j| z_2)$ and a polynomial. 

The last step of the proof is to establish smoothness with respect to $t$ and $x_1$. 
In short, it comes from the fact that $t$- and $x_1$-derivatives of $\psi^k_\bl, \phi^k_\bl$ satisfy the same kind of equations as $\psi^k_\bl, \phi^k_\bl$ themselves, and so the same kind of estimates. For instance, $\partial_{x_{1}}\psi^k_\bl$ satisfies formally
\begin{equation*}
\left\{
\begin{aligned}
\Delta_z \partial_{x_{1}}\psi_\bl^k & = - \pa^2_{x_1} v_{\bl,1}^{k - N_0}, \quad z \in \Omega_\bl, \\
n \cdot \na_z \partial_{x_{1}}\psi_\bl^k & = - \sum_{j=0}^{k-1}B_{k-j}[\partial_{x_{1}}u^j] -\frac{\partial_{x_{1}}h^k}{ \langle \eps^\alpha \eta'(z_{1})\rangle }, \quad z \in \pa \Omega_\bl. 
\end{aligned}
\right.
\end{equation*}
We refer to \cite{OleYos1,OleYos2,NeussMikelic} for more on related problems.
This concludes the construction of the boundary layer correctors.
\end{proof}
\begin{remark}
We insist that $\Omega_\bl$ depends on $\eps$, but as this dependence is regular, the constant $C$ in the estimate of the last proposition does not depend on $\eps$. 
\end{remark}

\subsection{Construction of the interior profiles}
From the analysis of the previous paragraph, we see that if the boundary layer profiles $v^j$ and interior profiles $u^j$ are given for $j \le k-1$, one can construct $v^k_\bl$ using formula \eqref{stream_potential} and systems \eqref{psikbl}--\eqref{phikbl}. In order to close the iterative process, we still need to explain how to build the interior profile $u^k$.

\medskip
 From the previous paragraph, we know that $u^k$ should satisfy the boundary condition \eqref{BCint}, related to the introduction of $h_k$ in \eqref{hk}. Furthermore, if we plug the expansion $\sum \eps^{\alpha k} u^k$ in the momentum equation, we end up with the following system in the flat domain: for all $k \ge 1$, 
\begin{equation} \label{eq_uk}
\left\{
\begin{aligned}
\pa_t u^k + u^0 \cdot \nabla u^k + u^k \cdot \nabla u^0 + \na p^k & = - \sum_{j=1}^{k-1} u^j \cdot \nabla_x u^{k-j}, \quad \mbox{in} \: \Omega_0, \\
{\rm div} \, u^k & = 0, \quad \mbox{in} \: \Omega_0,\\
u^k_2\vert_{x_2 = 0} & = h^k, \\
u^k\vert_{t < 0} & = 0, \quad \mbox{in} \: \Omega_0.
\end{aligned}
\right.
\end{equation} 
Again, the resolution of \eqref{eq_uk} is performed inductively on $k$. A necessary and sufficient condition for existence and uniqueness of a solution in $C^\infty(\R, H^\infty(\Omega_0))$ is 
\begin{equation} \label{compatibility_hk} 
\int_{\TT} h^k(t,x_1) dx_1 = 0. 
\end{equation}
More precisely, one can under condition \eqref{compatibility_hk} find a smooth $\tilde u^k$, compactly supported in variable $x_2$, such that 
\begin{equation*}
{\rm div} \, \tilde u^k = 0, \quad \tilde u_2^k\vert_{x_2=0} = h^k. 
\end{equation*}
See \cite[section III.3]{Galdi} for details. Hence, $U^k = u^k - \tilde u^k$ satisfies 
\begin{equation*}
\left\{
\begin{aligned}
\pa_t U^k + u^0 \cdot \nabla U^k + U^k \cdot \nabla u^0 + \na P^k & = F^k, \quad \mbox{in} \: \Omega_0, \\
{\rm div} \, U^k & = 0, \quad \mbox{in} \: \Omega_0,\\
U^k_2\vert_{x_2 = 0} & = 0, \\
U^k\vert_{t < 0} & = 0, \quad \mbox{in} \: \Omega_0,
\end{aligned}
\right.
\end{equation*}
with $F_k = - \sum_{j=1}^{k-1} u^j \cdot \nabla_x u^{k-j} - u^0 \cdot \na \tilde u^k - \tilde u^k \cdot \na u^0$. Finally, one shows global well-posedness of this linearized Euler system with impermeability condition (say in $H^s$ for arbitrary $s$). We do not give further details, and refer to \cite{KatoLai} for the more complex case of the nonlinear Euler equations in smooth bounded domains. 

\medskip
It remains to show that the compatibility condition \eqref{compatibility_hk} holds. From expression \eqref{hk}, we compute 
\begin{align*}
\int_\TT h_k(t,x_1) dx_1 & = \int_\TT \left( -\int_{\Omega_\bl} \pa_{x_1} v_{\bl,1}^{k - N_0}(t,x_1,z) dz - \int_{\pa \Omega_\bl} \sum_{j=0}^{k-1}B_{k-j}[u^j](t,x_1,z_1) d\sigma(z) \right) dx_1 \\
& = - \int_\TT \int_{\pa \Omega_\bl} \sum_{j=0}^{k-1}B_{k-j}[u^j](t,x_1,z_1) d\sigma(z) dx_1.
 \end{align*} 
We now recall that for any smooth divergence-free $u = u(t,x)$, expression ${ \sum \eps^{\alpha i} \langle \varepsilon^\alpha \eta'\rangle B_i[u](t,x_1,z_1)}$ comes from the Taylor expansion of $ \langle \varepsilon^\alpha \eta'\rangle u(t,x_1,\eps^{1+\alpha} \eta(z_1)) \cdot n^\eps(z_1)$ in variable $x_2 = \eps^{1+\alpha} \eta(z_1)$; see \eqref{exp-Bv}. To show that 
$$ \int_\TT \int_{\pa \Omega_\bl} B_i[u](t,x_1,z_1) d\sigma(z) dx_1= \int_\TT \int_\TT B_i[u](t,x_1,z_1)\langle \varepsilon^\alpha \eta'\rangle dz_{1} dx_1 = 0 $$
it is then enough to show that
$$ \int_\TT \int_\TT \langle \varepsilon^\alpha \eta'(z_{1})\rangle u(t,x_1,\eps^{1+\alpha} \eta(z_1)) \cdot n^\eps(z_1) dz_{1} dx_1 = \int_\TT \int_{\pa \Omega_\bl} u(t,x_1,\eps^{1+\alpha} \eta(z_1)) \cdot n^\eps(z_1) d\sigma(z) dx_1 = 0.
$$ 
But we can write 
\begin{align*}
 & \int_\TT \int_{\pa \Omega_\bl} u(t,x_1,\eps^{1+\alpha} \eta(z_1)) \cdot n^\eps(z_1) d\sigma(z) dx_1 \\ 
 = & -\int_\TT \int_{\Omega_\bl} \divz ( z \rightarrow u(t,x_1,\eps z_2) ) d z dx_1 \\
 = & -\eps \int_\TT \int_{\Omega_\bl} \pa_2 u_2(t,x_1,\eps z_2) dz dx_1 = \eps \int_\TT \int _{\Omega_\bl} \pa_1 u_1(t,x_1,\eps z_2) dz dx_1 = 0. 
\end{align*} 
The compatibility condition \eqref{compatibility_hk} is therefore satisfied, which ensures the well-posedness of systems \eqref{eq_uk}.

\subsection{Proof of Proposition~\ref{prop-Euler}}
We can now conclude the proof of the proposition. Let $s,M$ be large, and $(u^0,p^0)$ be a smooth solution of Euler in the flat domain $\Omega^0$. Following the analysis of the previous paragraph, we set 
$$ u^\app(t,x) = u^0(t,x) + u^\app_\inte(t,x) + u^\app_\bl(t,x) $$
where, for some arbitrary large $N$: 
$$ u^\app_\inte = \sum_{k=1}^N \eps^{\alpha k} u^k(t,x), \quad u^\app_\bl(t,x) = v^\app_\bl(t,x_1,x/\eps),$$ 
$$ \quad v^\app_\bl(t,x_1,z) = \sum_{k=1}^N \eps^{\alpha k} v^k_\bl(t,x_1,z) = \sum_{k=1}^N \eps^{\alpha k} \left( \na_z \psi^k_\bl + \na_z^\perp \phi^k_\bl \right)(t,x_1,z).$$
Let us note here that the first terms in $u^\app_\inte$ are zeros. Indeed, we recall that for any smooth $u$, 
 $B_k[u] = 0$ for $1 < k \le N_0$.
 Moreover, taking $k = 1$ in \eqref{hk}, we get 
\begin{equation*}
h^1(t,x_1) = - \int_{\pa \Omega_\bl} B_1[u^0](t,x_1,z_1) d\sigma(z) = u^0_1(t,x_1,0) \int_{\TT} \eta'(z_1) d z_1 = 0.
\end{equation*}
 From this, together with consideration of \eqref{hk}--\eqref{psikbl} and \eqref{eq_uk}, it follows that 
\begin{equation}\label{psivanishes}
 \psi^k_{\bl} = 0 \quad \forall 1 < k \le N_0, \quad u^k = 0 \quad \forall 1 \le k \le N_0. 
\end{equation}
Also, from \eqref{phikbl}, we deduce that 
$$ \phi^k_{\bl} = 0 \quad \forall 1 \le k \le N_0.$$
Hence, 
\begin{gather*}
 u^\app_\bl(t,x) = \eps^\alpha \na_z \psi^1_\bl(t,x_1,\frac{x}{\eps}) +\sum_{k=N_{0}+1}^N \eps^{\alpha k} v^k_\bl(t,x_1,\frac{x}\eps)\\
 u^\app_\inte(t,x) = \sum_{k=N_{0}+1}^N \eps^{\alpha k} u^k(t,x).
 \end{gather*}

As expected from the construction, the non-penetration at the boundary is almost satisfied by $u^\app$, in the following sense: 
$$ \| u^\app \cdot n^\eps \|_{W^{s',\infty}(\partial\Omega^\eps)} = \eps^{\alpha (N + 1) - s'}, \quad \forall s'. $$
The loss of a factor $\eps^{-1}$ with each order of derivation comes from the boundary layer. Similarly, the divergence-free condition is almost satisfied, in the sense that 
$$ \| {\rm div}~ u^\app \|_{W^{s',\infty}(\Omega^\eps)} = \eps^{\alpha (N+1) - 1 - s'}, \quad \forall s'. $$

\medskip
By standard arguments, see \cite[section III.3]{Galdi}, it is then possible to correct these small inhomogeneous terms: one can add a small corrector $\tilde u^\app$ so that 
$u^\app + \tilde u^\app$ is divergence-free and tangent at the boundary. Moreover, taking $N$ large enough, one can ensure that the source term created in the momentum equation by this additional corrector is arbitrarily small in $H^s$. For brevity, we do not discuss further this point, and consider that 
\begin{equation*}
 u^\app \cdot n^\eps = 0 \quad \mbox{ at } \: \pa \Omega^\eps, \quad {\rm div}~ u^\app = 0 \quad \mbox{in } \: \Omega^\eps. 
 \end{equation*} 

\medskip
It remains to check inequalities \eqref{all_bounds}. The first and second inequalities follow directly from the structure of the corrector. The third one ($L^2$ bound for the boundary layer derivatives) follows from the following basic lemma:
\begin{lemma}\label{classicallem}
 Let $\gamma$ be fixed. Then, there exists $C$ independent of $\varepsilon$ such that
 \[
 \| x\mapsto e^{\gamma x_2/\eps} f(x_{1},\frac{x}\eps) \|_{L^2(\Omega^\varepsilon)} \leq C\varepsilon^{1/2} \| e^{\gamma z_{2}} f \|_{L^2_{x}(\mathbb{T}, H^1_{z}(\Omega_{\bl}))}
 \]
 for all $f=f(x_{1},z)$.
\end{lemma}
\begin{proof}
 We compute
\begin{align*}
\| e^{\gamma x_2/\eps} f(x_{1},\frac{x}\eps) \|_{L^2(\Omega^\varepsilon)}^2 =& \int_{0}^1 \int_{\varepsilon^{\alpha+1}\eta(x_{1}/\eps)}^{+\infty} \Big| f(x_{1},\frac{x_{1}}\eps,\frac{x_{2}}\eps) \Big|^2 e^{2\gamma x_{2}/\eps} \, dx_{2}dx_{1}\\
=&\varepsilon \int_{0}^1 \int_{\varepsilon^{\alpha}\eta(x_{1}/\eps)}^{+\infty} \Big| f(x_{1},\frac{x_{1}}\eps,z_{2}) \Big|^2 e^{2\gamma z_{2}} \, dz_{2}dx_{1}\\
\leq & \varepsilon \int_{0}^1 \sup_{z_{1}\in \mathbb{T}} 
\Bigg[ \underbrace{ \int_{\varepsilon^{\alpha}\eta(z_{1})}^{+\infty} \Big| f(x_{1},z_{1},z_{2}) \Big|^2 e^{2\gamma z_{2}} \, dz_{2} }_{=: F_{x_{1}}(z_{1})} \Bigg] \,dx_{1} .
\end{align*}
As $W^{1,1}(\mathbb{T})$ embeds in $L^{\infty}(\mathbb{T})$, we write
\begin{align*}
 \| F_{x_{1}} \|_{L^{\infty}(\mathbb{T})} \leq C \Big(& \| F_{x_{1}} \|_{L^{1}(\mathbb{T})} + \| \partial_{z_{1}} F_{x_{1}} \|_{L^1(\mathbb{T})}\Big)\\
 \leq C \Big(& \| e^{\gamma z_{2}} f \|_{L^2_{z}(\Omega_{\bl}))}^2 + \int_{\Omega_{\bl}} 2 | f(x_{1},z_{1},z_{2}) | | \partial_{z_{1}} f(x_{1},z_{1},z_{2}) | e^{2\gamma z_{2}} \, dz\\
 &+ \int_{0}^1 \varepsilon^\alpha |\eta'(z_{1})| | f(x_{1},z_{1},\varepsilon^{\alpha}\eta(z_{1})) |^2 e^{2\gamma \varepsilon^{\alpha}\eta(z_{1})}\, dz_{1} \Big)\\
 \leq C \Big(& 2 \| e^{\gamma z_{2}} f \|_{L^2_{z}(\Omega_{\bl}))}^2 + \| e^{\gamma z_{2}} \partial_{z_{1}} f \|_{L^2_{z}(\Omega_{\bl}))}^2 + \| e^{\gamma z_{2}} f \|_{L^2_{z}(\partial\Omega_{\bl}))}^2\Big).
\end{align*}
The embedding of $H^1(\Omega_{\bl})$ into $L^2(\partial\Omega_{\bl})$ gives the result.
\end{proof}
To obtain the fourth inequality in \eqref{all_bounds}, we notice that 
$$ \curl u^\app_\bl = -\varepsilon^{\alpha(N+1-N_{0})}\partial_{x_{1}} v_{\bl,1}^{N+1-N_{0}} \: + \: \text{ lower order terms},$$ 
so that 
$$ \sup_{t\in [0,T_0]} \| e^{\gamma x_2/\eps} \eps^{|\beta|} \pa^\beta \curl u^\app_\bl\|_{L^\infty(\Omega^\eps)} \le C_0 \epsilon^{\alpha(N+1)-1} \lesssim \eps^M$$ 
for $\alpha(N+1)-1 \ge M$.
Eventually, we have to estimate the source term due to the approximation. We write the nonlinearity as 
$$ u \cdot \na u = (\curl u) \, u^\perp + \na \frac{|u|^2}{2}, \quad \text{with as usual } \curl u=\partial_{x_{1}} u_{2} -\partial_{x_{2}} u_{1} \text{ and } u^\perp=(-u_{2},u_{1}).$$
We get in particular, for some pressure $p$:
\begin{multline}
 \pa_t u^\app + \curl u^\app (u^\app)^\perp + \na p \\
= \pa_t u^\app_\bl + (\curl u_0 + \curl u^\app_\inte) (u^\app_\bl)^\perp + \curl u^\app_{\bl} (u^\app)^\perp + R^\app_\inte \label{remainder}
\end{multline}
where 
$$R^\app_\inte = \pa_t (u^0 + u^\app_\inte) + (u^0 + u^\app_\inte) \cdot \na (u^0 + u^\app_\inte) + \na (p^0 + p^\app_\inte) $$
 satisfies by construction the fifth estimate in \eqref{all_bounds}. Also, 
\begin{align*}
\| \curl u^\app_{\bl} (u^\app)^\perp \|_{W^{s,\infty}} &\leq C \| \curl u^\app_{\bl} \|_{W^{s,\infty}} ( \| u^0\|_{W^{s,\infty}} + \|u^\app_{\inte} \|_{W^{s,\infty}} + \| u^\app_{\bl} \|_{W^{s,\infty}}) \\
&\leq C \eps^{\alpha (N+1)-1-s} ( 1 + \eps^{\alpha+1} + \eps^{\alpha-s}) \leq C_{0} \varepsilon^{\alpha(N+2) -2s}.
\end{align*}
For $s,M$ given, we can choose $N$ large enough such that $\alpha(N+2) -2s\geq M$, and hence we can include $\curl u^\app_{\bl} (u^\app)^\perp $ in the definition of $R^\app_\inte$.
In order to estimate the other part in the right-hand side of \eqref{remainder}, we use the previous estimates together with 
$$\curl u_\inte^\app(t,x) = \mathcal{O}(\eps^{1 + \alpha}) \quad \mbox{in} \: W^{s,\infty}(\Omega^\eps), $$
and with the product rule
\begin{equation*}
 \|fg \|_{H^s_{\varepsilon,\gamma}} \lesssim \| f \|_{W^{s,\infty}} \| g \|_{H^s_{\varepsilon,\gamma}}, \quad \| g \|_{H^s_{\eps,\gamma}} \: := \: \sum_{|\beta| \le s} \| e^{\gamma x_2/\eps} \eps^{|\beta|} \pa^\beta g \|_{L^2(\Omega^\eps)} 
 \end{equation*}
to deduce 
\begin{align*}
 \pa_t u^\app_\bl(t,x) +& (\curl u^0 + \curl u^\app_\inte) (u^\app_\bl)^\perp(t,x) \\
= & \eps^\alpha \pa_t \na_z \psi^1_\bl(t,x_1,x/\eps) + \eps^\alpha \curl u^0 (t,x) \na_z^\perp \psi^1_\bl(t,x_1,x/\eps) + \mathcal{O}(\eps^{\alpha+\frac32}) \quad \mbox{in} \: H^s_{\eps,\gamma} \\
= & \eps^\alpha \pa_t \na_z \psi^1_\bl(t,x_1,x/\eps) + \eps^\alpha \curl u^0(t,x_1,0) \na_z^\perp \psi^1_\bl(t,x_1,x/\eps) \\ 
 & + \eps^\alpha x_2 \left( \int_0^1 \pa_2 \curl u^0(t,x_1,s x_2) ds \right) \na_z^\perp \psi^1_\bl(t,x_1,x/\eps) + \mathcal{O}(\eps^{\alpha+\frac32}) \quad \mbox{in} \: H^s_{\eps,\gamma} \\
= & \eps^\alpha \pa_t \na_z \psi^1_\bl(t,x_1,x/\eps) + \eps^\alpha \curl u^0(t,x_1,0) \na_z^\perp \psi^1_\bl(t,x_1,x/\eps) \\
 & + \eps^{1+\alpha} \left( \int_0^1 \pa_2 \curl u^0(t,x_1,s x_2) ds \right) \left( z_2 \na_z^\perp \psi^1_\bl(t,x_1,z)\right)\vert_{z = x/\eps} + \mathcal{O}(\eps^{\alpha+\frac32}) \quad \mbox{in} \: H^s_{\eps,\gamma} .
\end{align*}
For $\tilde \gamma \in (\gamma,1)$, there exists $C$ such that 
$$\| e^{\gamma z_{2}} z_2 \na_z^\perp \psi^1_\bl(t,x_1,z) \|_{H^s(\Omega_{\bl})}\leq C \| e^{\tilde\gamma z_{2}} \na_z^\perp \psi^1_\bl(t,x_1,z) \|_{H^s(\Omega_{\bl})} .$$
Hence, by Lemma~\ref{classicallem}, we have 
\[
\pa_t u^\app_\bl(t,x) + (\curl u^0 + \curl u^\app_\inte) (u^\app_\bl)^\perp(t,x) = \varepsilon^\alpha \tilde v(t,x_1,z) + \mathcal{O}(\eps^{\alpha+\frac32}) \quad \mbox{in} \: H^s_{\eps,\gamma} 
\]
with 
$$ \tilde v(t,x_1,z) = \pa_t \na_z \psi^1_\bl + \curl u^0(t,x_1) \na_z^\perp \psi^1_\bl(t,x_1,z) .$$
To conclude, we notice that $\curl_z \tilde v = 0$, because $\psi^1_\bl$ is harmonic in variable $z$. We can thus write 
$$ \tilde v(t,x_1,z) = \na_z q(t,x_1,z) $$
for a scalar function $q$. The fact that $q$ is periodic in $z_1$ follows from the fact that 
$$ F_{x_{1}}(z_{2}):= \int_\TT \pa_{z_2} \psi^1_\bl(t,x_1,z_1, z_2) dz_1 = 0 \quad \mbox{ for any } \: z_2 > \sup \eps^\alpha \eta $$
because $\partial_{z_{2}} F_{x_{1}}(z_{2}) = -\int_\TT \pa^2_{z_1} \psi^1_\bl(t,x_1,z_1, z_2) dz_1=0$ and that $\lim_{z_{2}\to \infty}F_{x_{1}}(z_{2}) =0$. Hence, 
$$ \tilde v(t,x_1,x/\eps) = \eps \na_x \left( q(t,x_1,x/\eps) \right) - \eps \pa_{x_1} (q, 0)(t,x_1,x/\eps). $$
Writing $e^{\gamma z_{2}}f(z_{2}) =-e^{\gamma z_{2}}\int_{z_{2}}^\infty f'(s)\,ds =-\int_{\R} h(z_{2}-s) e^{\gamma s}f'(s)\,ds $, with $h(t)=e^{\gamma t} \mathds{1}_{\R^-}(t)$, we deduce the following estimates: for all $p \in [1,+\infty]$,
$$ \| e^{\gamma z_{2}} f \|_{L^p(\varepsilon^\alpha \eta(z_{1}),\infty)} \leq \frac1\gamma \| e^{\gamma z_{2}} f' \|_{L^p(\varepsilon^\alpha \eta(z_{1}),\infty)}.$$
Therefore, we have $ \| e^{\gamma z_{2}} \pa_{x_1}q(t,x_{1},z) \|_{L^p_{z}(\Omega_\bl)} \lesssim \| e^{\gamma z_{2}} \pa_{x_1} \tilde v_{2} (t,x_{1},z) \|_{L^p_{z}(\Omega_\bl)}$. Differentiating more in $x_{1}$ gives similar estimates, and the derivative with respect of $z_{1}$ and $z_{2}$ are even simpler as $\partial_{z_{i}}q = \tilde v_{i}$. Hence, 
$$ \| e^{\gamma x_2/\eps} \eps^{\beta} \pa_x^\beta \pa_{x_1} q(t,x_1,x/\eps) \|_{L^\infty(\Omega^\eps)} \lesssim 1, \quad \|  e^{\gamma x_2/\eps} \eps^{\beta} \pa_x^\beta \pa_{x_1} q(t,x_1,x/\eps) \|_{L^2(\Omega^\eps)} \lesssim \eps^{1/2}. $$
To conclude, replacing $p $ in \eqref{remainder} by $p^\app(t,x) = p(t,x) - \eps^{1+\alpha} q(t,x,x/\eps)$, we find that 
$$ \pa_t u^\app + \curl u^\app (u^\app)^\perp + \na p^\app = R^\app_\inte + R^\app_\bl $$
where $R^\app_{\bl}$ satisfies the bounds of the proposition. Proposition~\ref{prop-Euler} is thus proved.

\section{Stability estimates}
Let $T_0 > 0$. Let $u^{\nu, \varepsilon}(t,x)$ be the solution to the Navier-Stokes system \eqref{NS}, and let us introduce $v$ through
$$ u^{\nu, \varepsilon}(t,x) = u^\app(t,x) + \mathsf{v}(t,x),$$
in which $u^\app$ is the Euler approximate solution constructed in Proposition~\ref{prop-Euler} (Section~\ref{sec-Euler}). Then, it is immediate that the perturbation $\mathsf{v}$ to our approximation $u^\app$ solves 
\begin{subequations}\label{prob-v}
\begin{align} 
\mathsf{v}_t + (u^\mathrm{app} + \mathsf{v})\cdot \nabla \mathsf{v} + \mathsf{v}\cdot \nabla u^\mathrm{app} + \nabla p - \nu \Delta \mathsf{v} &= R^{\nu, \varepsilon}_\mathrm{app}
\label{NSv-eqs} \\
\Div \mathsf{v} &= 0 \label{NSdiv-eqs} 
\\
\mathsf{v} \cdot n^\e &= 0 \qquad \mbox{on}\quad \pa \Omega^\epsilon \label{NSbdry-eqs}
\\
2D(\mathsf{v})n^\epsilon \cdot \tau^\epsilon + \lambda \mathsf{v} \cdot \tau^\eps&= - 2D(u^\app)n^\epsilon \cdot \tau^\epsilon- \lambda u^{\app} \cdot \tau^\eps \qquad \mbox{on}\quad \pa \Omega^\epsilon \label{NSbdry2-eqs}
\end{align}
\end{subequations}
in which we have denoted 
\begin{equation}\label{def-vRapp}
R^{\nu, \varepsilon}_\mathrm{app}: = \nu \Delta u^\mathrm{app} (t,x) - R^\mathrm{app}(t,x).
\end{equation}

One main issue to carry stability estimates is the singular dependence of the boundary layer part of $u^\app$ with respect to $\eps$: differentiation in $x$ leads to loss of powers of $\eps$. To avoid this difficulty, it is convenient to make the change of variables: 
$$ (t,x) \mapsto (\hat \tau,z):=\frac{(t,x)}{\epsilon} \quad \in\quad \tilde \Omega^\epsilon := \Big\{ z_1\in \TT_{\frac 1\e},~ z_2> \e^\alpha\eta(z_1) \Big\},$$
{\em and we will work throughout the paper with these new variables}. The main advantage of using variables $ (\hat \tau,z)$ is that differentiation in $z$ of our approximate solution $u^\app$ does not lose any power of $\eps$. {\it A contrario}, stability estimates that had to be established for $t \in [0,T_0]$ must now be established for $\hat{\tau} \in [0, T_0/\eps]$, that is on large time scales. 

\medskip
 In the domain $\tilde \Omega^\epsilon$ we use the functions $v(\hat \tau, z) = \mathsf{v}(t,x) $ and $\tilde u^\mathrm{app}(\hat \tau, z) = u^\mathrm{app}(t,x)$. Applying the change of variables, we see that $v$ solves 
\begin{subequations}\label{prob-resv}
\begin{align} 
v_{\hat \tau} + (\tilde u^\mathrm{app} + v)\cdot \nabla_z v + v\cdot \nabla_z \tilde u^\mathrm{app} + \nabla_z p - \tilde \nu \Delta_z v &= \epsilon R^{\nu, \varepsilon}_\mathrm{app}(\varepsilon\hat \tau, \varepsilon z)=\epsilon \tilde R^{\nu, \varepsilon}_\mathrm{app}(\hat \tau, z) \label{reNSv-eqs} \\
\Div_z~v &=0 \label{reNSdiv-eqs} \\
v \cdot n &= 0 \qquad \mbox{on}\quad \pa \tilde \Omega^\varepsilon \label{reNSbdry-eqs}\\
2D_z(v)n \cdot \tau+ \varepsilon\lambda v \cdot \tau &= - 2\epsilon D_x(u^\app)n \cdot \tau - \varepsilon\lambda u^\app \cdot \tau \ \mbox{on}\ \pa \tilde\Omega^\varepsilon \label{reNSbdry2-eqs}\\
&= - 2D_z(\tilde u^\app)n \cdot \tau - \varepsilon\lambda \tilde u^\app \cdot \tau \nonumber
\end{align}
\end{subequations}
where the new viscosity is defined by 
\begin{equation}\label{def-newvis} 
\tilde \nu = \frac \nu \ep.
\end{equation}
We also define the vorticities in rescaled variables,
\begin{equation}\label{curl-app}
\omega := \curl_z v, \quad \omega^\app:= \curl_z (\tilde u^\mathrm{app} )
 = \epsilon \curl u^0(\epsilon \hat \tau, \epsilon z) + \epsilon \curl u^\app_\inte(\e \hat \tau , \e z) + \epsilon \curl u^\app_\bl(\e \hat \tau , \e z) .
\end{equation}
Let us eventually mention that the stability estimates will be performed in the regime
\begin{equation}\label{nu-eps}
 \varepsilon^{N_{1}} \lesssim \tilde \nu \lesssim \varepsilon^6.
\end{equation}
 for $N_{1}\in \N^*$ fixed and arbitrary large.

\subsection{$L^2$ velocity estimates}

In this subsection, we prove the following standard energy estimate: 

\begin{proposition}\label{prop-L2} Assume that $\tilde \nu \lesssim 1$. There holds, for all $\hat{\tau} \in [0,T_0/\eps]$:
$$
\frac12 \frac{d}{d\hat \tau} \| v(\hat \tau)\|_{L^2(\tilde \Omega^\epsilon)}^2 + \tilde \nu \|\nabla v(\hat \tau) \|_{L^2(\tilde \Omega^\epsilon)}^2 \lesssim \e \| v(\hat \tau)\|_{L^2(\tilde \Omega^\epsilon)}^2 + \| v(\hat \tau)\|_{L^2(\tilde \Omega^\epsilon)} \| \omega(\hat \tau)\|_{L^2(\tilde \Omega^\epsilon)}
+ \e^{2\alpha + 2}+ \e^{2\alpha - 2} \tilde\nu^2 .$$
\end{proposition}

In view of a Gronwall lemma, let us note that the factor $\varepsilon$ in front of $\| v\|_{L^2(\tilde \Omega^\epsilon)}^2$ is crucial, because we are interested in uniform estimates for $\hat \tau \in [0,T_{0}/\varepsilon]$ which will imply uniform estimates for $t\in [0,T_{0}]$.
The rest of this subsection is dedicated to the proof of this proposition.

As usual for an energy estimate, we multiply the velocity equation by $v$ and the integration yields
\begin{equation}\label{est-L2}\begin{aligned}
\frac12 \frac{d}{d\hat \tau} \| v\|_{L^2(\tilde \Omega^\epsilon)}^2 + \tilde \nu \|\nabla v \|_{L^2(\tilde \Omega^\epsilon)}^2 &= \int_{\tilde \Omega^\epsilon} \Big[\epsilon v \cdot \tilde R^{\nu, \varepsilon}_{\mathrm{app}} - v\cdot (v\cdot \nabla \tilde u^\mathrm{app} ) \Big] -\tilde \nu \int_{\pa \tilde\Omega^\varepsilon} \partial_{n}v \cdot v d\sigma \\
&= \int_{\tilde \Omega^\epsilon} \Big[\epsilon v \cdot \tilde R^{\nu, \varepsilon}_{\mathrm{app}} - v\cdot (v\cdot \nabla \tilde u^\mathrm{app} ) \Big] -\tilde \nu \int_{\pa \tilde \Omega^\varepsilon} \partial_{n}v \cdot \tau \, v \cdot \tau d\sigma
 \end{aligned}\end{equation}
 where we recall that $n$ is the normal vector pointing inside the fluid domain and where we have used $v\cdot n=0$.

For the first term at the right-hand side, Proposition~\ref{prop-Euler} implies
\begin{equation}\label{Renuapp-L2}
\begin{aligned}
 \| \e \tilde R^{\nu,\e}_\mathrm{app} \| _{L^2 (\tilde \Omega^\epsilon)} = & \| R^{\nu,\e}_\mathrm{app} \| _{L^2 (\Omega^\epsilon)}\\
 \leq & \nu \| \Delta u^0 + \Delta u^{\app}_{\inte} \|_{L^2 (\Omega^\epsilon)} +\frac\nu{\varepsilon^2} \| \varepsilon^2 \Delta u^{\app}_{\bl} \|_{L^2 (\Omega^\epsilon)} + \| R^\mathrm{app} \| _{L^2 (\Omega^\epsilon)}\\
 \lesssim &\nu\varepsilon^{\alpha-\frac32} + \varepsilon^{\alpha+\frac32} \leq \varepsilon^{\frac12} ( \tilde\nu\varepsilon^{\alpha-1} +\varepsilon^{\alpha+1}). 
 \end{aligned}
 \end{equation}
 This bounds holds uniformly for $\hat{\tau} \in [0,\frac{T_0}{\eps}]$. It gives 
 \[
\Big| \int_{\tilde \Omega^\epsilon}\epsilon v \cdot \tilde R^{\nu, \varepsilon}_{\mathrm{app}} \Big| \lesssim \e \| v\|_{L^2(\tilde \Omega^\epsilon)}^2+ \e^{2\alpha + 2}+ \e^{2\alpha - 2} \tilde\nu^2 .
 \]

Concerning the second term, we use the relation
\[
a\cdot \nabla b + b\cdot \nabla a = \nabla (a\cdot b) + a^\perp \curl b + b^\perp \curl a
\]
to write
\begin{align*}
 \int_{\tilde \Omega^\epsilon} v\cdot (v\cdot \nabla \tilde u^\mathrm{app} ) \; dz &= \int_{\tilde \Omega^\epsilon} v\cdot \Big[ \nabla (v\cdot \tilde u^\mathrm{app}) - \tilde u^\mathrm{app} \cdot \nabla v + v^\perp \curl \tilde u^\mathrm{app} + (\tilde u^\mathrm{app})^\perp \curl v\Big] \; dz\\
 &=\int_{\tilde \Omega^\epsilon} v\cdot (\tilde u^\mathrm{app})^\perp \curl v \; dz,
\end{align*}
where we have used that $v$ and $\tilde u^\mathrm{app}$ are divergence-free and tangent to the boundary. Using that $\tilde u^\app$ is uniformly bounded (see \eqref{all_bounds}), we state
$$ \Big| \int_{\tilde \Omega^\epsilon} v\cdot (v\cdot \nabla \tilde u^\mathrm{app} ) \; dz \Big| \lesssim \| v\|_{L^2(\tilde \Omega^\epsilon)} \| \omega\|_{L^2(\tilde \Omega^\epsilon)} . $$ 
This bound, involving the vorticity, will reveal more useful than the direct one by $\| \nabla_z \tilde u^\app \|_{L^\infty} \| v \|_{L^2}^2$. Indeed, from the first two inequalities in \eqref{all_bounds}, $\| \na_z \tilde u^\app \|_{L^\infty(\tilde \Omega^\eps)} = \eps \| \na_x u^\app \|_{L^\infty(\Omega^\eps)} \lesssim \eps^{\alpha}$, and a bound by $\eps^{\alpha} \| v \|_{L^2}^2$ is not enough to conclude by a Gronwall's Lemma. 

To handle the boundary term in \eqref{est-L2}, we differentiate the boundary condition $v\cdot n=0$ tangentially to $\pa \tilde \Omega^\eps$. With $\tau = -n^\perp$, we find
$$ (\tau \cdot \na) v \cdot n + v \cdot (\tau \cdot \na) n = 0.$$
We then write $(\tau \cdot \na) n = - \kappa \tau$, where $\kappa$ is the algebraic curvature of $\pa \tilde \Omega^\eps$. We compute 
$$\begin{aligned}
 (\tau \cdot \na) v \cdot n &= (\na v)^T \tau \cdot n = \tau \cdot (\na v) n 
 \\&= 2 (D(v) n) \cdot \tau - \pa_n v \cdot \tau 
 \\&= - 2 (D(\tilde u^\app) n) \cdot \tau - \varepsilon \lambda (v+\tilde u^\app) \cdot \tau - \pa_n v \cdot \tau. \end{aligned}$$
Eventually, we obtain 
\begin{equation}\label{partialvn}
 \partial_{n} v \cdot \tau = -2 D_z (\tilde u^\app) n \cdot \tau- \varepsilon \lambda \tilde u^\app \cdot \tau -(\kappa+\varepsilon\lambda) v\cdot \tau.
\end{equation}
We compute
\begin{equation}\label{kappa}
 \kappa = \frac{ \varepsilon^\alpha \eta''}{\langle \varepsilon^\alpha \eta' \rangle^3}, \quad \text{then }\| \kappa \|_{L^{\infty}} \lesssim \varepsilon^\alpha .
\end{equation}
Moreover, in view of \eqref{all_bounds}:
\begin{equation} \label{Duapp_bord}
\begin{aligned}
 \| \tilde u^\app \| _{L^2(\pa \tilde \Omega^\varepsilon)} = \varepsilon^{-1/2} \| u^\app \| _{L^2( \pa \Omega^\varepsilon)} \lesssim \varepsilon^{-\frac12} \\
 \| D_z (\tilde u^\app) \| _{L^2(\pa \tilde \Omega^\varepsilon)} = \varepsilon^{1/2} \| D_x (u^\app) \| _{L^2( \pa \Omega^\varepsilon)} \lesssim \varepsilon^{\alpha-\frac12}.
\end{aligned}
\end{equation}
In the previous inequality and many times in the sequel, we use that for functions which are continuous up to the boundary, we have $ \| f \| _{L^\infty( \pa \Omega^\varepsilon)} \leq \| f \| _{L^\infty( \Omega^\varepsilon)} $. Note again that the previous inequalities hold uniformly for $\hat{\tau} \in [0,\frac{T_0}{\eps}]$. Assuming $|\lambda | \lesssim \varepsilon^{-1+\alpha}$, this implies
\begin{align*}
\Big| \tilde \nu \int_{\pa \tilde\Omega^\varepsilon} \partial_{n}v \cdot \tau \, v \cdot \tau d\sigma \Big|
&\lesssim \tilde \nu\Big( \| v\cdot \tau \|_{L^2(\pa \tilde\Omega^\varepsilon)}\varepsilon^{\alpha-\frac12} +\varepsilon^\alpha \| v\cdot \tau \|^2_{L^2(\pa \tilde\Omega^\varepsilon)} \Big)\\
&\lesssim (1+ \tilde\nu \varepsilon^\alpha) \| v\cdot \tau \|^2_{L^2(\pa \tilde\Omega^\varepsilon)} + \tilde\nu^2\varepsilon^{2\alpha-1}
\end{align*}
which gives (by a trace lemma: Lemma~\ref{lem-trace2})
\[
\Big| \tilde \nu \int_{\pa \tilde\Omega^\varepsilon} \partial_{n}v \cdot \tau \, v \cdot \tau d\sigma \Big|
\lesssim \| v\|_{L^2(\tilde \Omega^\epsilon)} \| \omega\|_{L^2(\tilde \Omega^\epsilon)} + \tilde \nu^2\varepsilon^{2\alpha-1}
\]
where we have used $\tilde \nu \lesssim 1$. Note that we could also estimate this boundary term in a more classical
way by using the trace lemma~\ref{lem-trace} and the energy dissipation term in the estimate \eqref{est-L2}. This ends the proof of Proposition~\ref{prop-L2}.

\subsection{$L^2$ vorticity estimate}

From the estimate of Proposition~\ref{prop-L2}, it is clear that we need an estimate for the $L^2$ norm of the vorticity $\omega = \curl_z v$. We first observe that it solves 
\begin{equation}\label{eqs-resw}
\partial_{\hat \tau} \omega + (\tilde u^\mathrm{app} + v)\cdot \nabla_z \omega + v\cdot \nabla_z \omega^\mathrm{app} = \tilde \nu \Delta \omega + \e \curl_z \tilde R^{\nu, \varepsilon}_\mathrm{app}.
\end{equation}
Moreover, by writing 
$$ \pa_n v \cdot \tau = D(v)n \cdot \tau + \frac{1}{2} \omega n^\perp \cdot \tau = - D(\tilde u^\app)n \cdot \tau - \frac{\varepsilon \lambda}2 (v+\tilde u^\app) \cdot \tau - \frac{1}{2} \omega, $$
identity \eqref{partialvn} yields the Dirichlet condition 
\begin{equation}\label{BC-resw}
\omega = 2 D_z (\tilde u^\app) n \cdot \tau + \varepsilon \lambda \tilde u^\app \cdot \tau + (2 \kappa+\varepsilon\lambda) v\cdot \tau .
\end{equation} 
 
As we have 
$$\begin{aligned}
 \e \curl_z \tilde R^{\nu, \varepsilon}_\mathrm{app} 
& = \e \curl_z\Big( \nu \Delta_x u^\mathrm{app} (\e \hat \tau, \e z) - R^\mathrm{app}(\e\hat \tau,\e z)\Big) 
\\&= \e^2 \curl_x\Big( \nu \Delta_x u^\mathrm{app} - R^\mathrm{app}\Big) (\e \hat \tau, \e z),
\end{aligned}$$
we deduce from Proposition~\ref{prop-Euler} that
\begin{align}
\| \e \curl_z \tilde R^{\nu, \varepsilon}_\mathrm{app} \|_{L^2_z(\tilde \Omega^\e)} 
&\lesssim \nu \varepsilon \| D^2_x \curl_x u^\app \|_{L^2_x( \Omega^\e)} +\varepsilon \| D_x R^\app \|_{L^2_x( \Omega^\e)} \nonumber \\
&\lesssim \nu \varepsilon \| u^0 + u^\app_{\inte} \|_{H^3( \Omega^\e)}+\nu \varepsilon^{-1} \| \eps^2 D^2_x \curl_x u^\app_\bl \|_{L^2( \Omega^\e)}+ \varepsilon \| R^\app_{\inte} \|_{H^1( \Omega^\e)} \nonumber\\ & + \| \eps D_x R^\app_{\bl} \|_{L^2( \Omega^\e)}
\lesssim \tilde\nu\varepsilon^2 + \e^{\alpha+\frac32} , \label{L2-Rappw} 
\end{align}
which is in particular smaller than $\e^{\alpha+\frac32}$ (for instance, when $\tilde \nu\lesssim 1$ and $\alpha\leq 1/2$, or $\tilde \nu\lesssim \varepsilon^{1/2}$).

In this section, we shall derive the following key estimates. 

\begin{proposition}\label{prop-vort} Assume that $\tilde \nu \lesssim \varepsilon^4$. 
If 
\begin{equation}\label{assmp-mT}
\sup_{0\le \hat \tau\le T/\varepsilon}\|\nabla_z v (\hat \tau)\|_{L^\infty(\tilde \Omega^\epsilon)}
 \leq 1\quad \text{with} \quad T\leq T_{0},
\end{equation}
then, we have (for $\varepsilon$ small enough) the uniform velocity bound 
\begin{equation} \label{vLinftyL2}
\sup_{0\le \hat \tau \leq T/\varepsilon} \e \| v (\hat \tau)\|_{L^2(\tilde \Omega^\epsilon)} + \tilde \nu^{\frac12} \eps \| \nabla v \|_{L^2((0,T/\varepsilon)\times \tilde \Omega^{\e})}\lesssim \e^{\alpha+\frac12}+ \tilde \nu^{1/4} \e^{\alpha - 1} ,
\end{equation}
and the uniform vorticity bounds 
\begin{gather*}
 \sup_{0\le \hat \tau \leq T/\varepsilon} \| \omega (\hat \tau )\|_{L^2(\{z_2 - \e^\alpha \eta(z_1)\gtrsim \sqrt{\tilde \nu}\})} \lesssim \e^{\alpha+\frac12}+ \tilde \nu^{1/4} \e^{\alpha - 1}, \\
 \|\omega\|_{L^2(0, T/\varepsilon; L^2(\tilde \Omega^\epsilon))} 
\lesssim 
 \e^{\alpha}+ \tilde \nu^{1/4} \e^{\alpha - \frac32}.
\end{gather*}
\end{proposition}
\begin{remark}
Note that the condition $\tilde \nu \lesssim \varepsilon^4$ implies $\e \| v \|_{L^2} \ll 1$. The proposition and the fact that $\e \| v \|_{L^2} \ll 1$ actually hold under the weaker assumption $\tilde \nu \ll \varepsilon^{4-\frac\alpha4}$. As it is not a significant improvement ($\alpha=1/N_{0}$ can be arbitrary small), we keep the assumption $\tilde \nu \lesssim \varepsilon^4$ for simplicity.
\end{remark}

The rest of this section is dedicated to the proof of Proposition~\ref{prop-vort}.

\subsubsection{Weighted estimates}

Let $\phi = \phi(z)$ be some non-negative and bounded weight function, to be determined below, so that $\phi =0$ on the boundary $\Gamma $. We shall use $\phi \omega$ as a test function. Note that 
there is no boundary condition term appearing, when taking integration by parts with $x_1$ derivatives due to the periodicity assumption. Thus, multiplying the vorticity equation by $\phi \omega$, we get 
\begin{equation}\label{id-L2w}\begin{aligned}
 \frac 12 &\frac{d}{d\hat \tau} \int \phi |\omega|^2 + \tilde \nu \int \phi |\nabla \omega|^2 
 \\&= \int\Big[ \frac 12(\tilde u^\mathrm{app} + v)\cdot (\nabla \phi) |\omega|^2 - \phi \omega v \cdot \nabla \omega^\mathrm{app} - \tilde \nu \nabla \phi \cdot \nabla \omega \omega + \e \phi \omega\curl \tilde R^{\nu, \varepsilon}_\mathrm{app} \Big] 
 \\&= \int\Big[ (\tilde u^\mathrm{app} + v)\cdot \nabla \phi+\tilde \nu \Delta\phi \Big] \frac{|\omega|^2}{2} - \int\Big[ \phi \omega v \cdot \nabla \omega^\mathrm{app} - \e\phi \omega\curl \tilde R^{\nu, \varepsilon}_\mathrm{app} \Big] + \frac {\tilde \nu} 2 \int_{\pa \tilde \Omega^\eps } \frac{\partial \phi}{\partial n} |\omega|^2.
 \end{aligned}\end{equation}
The most dangerous term is the convection term and we set 
\begin{equation}\label{def-Gnorm}
\begin{aligned}
m(T): &= 1+ \sup_{0\le \hat \tau\le T/\varepsilon}\|\nabla_z (\tilde u^\app + v)(\hat \tau)\|_{L^\infty(\tilde \Omega^\epsilon)} .
\end{aligned}
\end{equation} 
This is where the choice of our weight function $\phi$ comes in.

\begin{remark}\label{rem-mT}
 As $\|\nabla_z \tilde u^\app \|_{L^\infty(\tilde \Omega^\epsilon)}=\varepsilon\|D_x u^\app \|_{L^\infty( \Omega^\epsilon)}$, it follows from \eqref{all_bounds} and assumption \eqref{assmp-mT} that $m(T)\lesssim 1$.
\end{remark}

\begin{lemma}[Construction of weight functions]\label{lem-phi}
If $m(T) \lesssim 1$, there exists a non-negative weight function $\phi = \phi(z)$ so that 
$$ \phi(z) =0, \qquad \mbox{on}\quad \pa \tilde \Omega^\eps = \{z_{1}\in \mathbb{T}_{\frac1\varepsilon},\ z_2 = \e^\alpha \eta(z_1)\}$$
 $$ \phi \lesssim \sqrt{\tilde \nu}, \qquad |\nabla \phi|\lesssim 1, \qquad (\tilde u^\mathrm{app} + v) \cdot \nabla \phi+ \tilde \nu \Delta\phi \lesssim \tilde \nu \e^\alpha \: \text{ over } (0,T),$$
in $ \tilde \Omega^\epsilon$. In addition, when $z_2 - \e^\alpha\eta(z_1) \gtrsim \sqrt {\tilde \nu} $, we have $\phi(z) \gtrsim \sqrt{\tilde \nu}.$
\end{lemma}

\begin{proof} We recall that on $\pa \tilde \Omega^\eps $ the normal and tangential directions are defined by 
\begin{equation}\label{def-ntau}
 \tau = \frac{(1,\e^\alpha \eta'(z_1))}{\langle \varepsilon^\alpha \eta'\rangle }, \qquad n = \frac{(-\e^\alpha \eta'(z_1), 1)}{\langle \varepsilon^\alpha \eta'\rangle } , \qquad \langle \varepsilon^\alpha \eta'\rangle =\sqrt{1+\e^{2\alpha }|\eta'(z_1)|^2},
 \end{equation}
in which $\eta(z_1)$ is one-periodic in $z_1$. We shall work with the new variables:
\begin{equation}\label{def-newz} \tilde z_1 = z_1, \qquad \tilde z_2 = z_2 - \e^\alpha \eta(z_1).\end{equation}
In this new variables, we note that 
\begin{equation*}
\nabla_z = \begin{pmatrix} 1& -\e^\alpha \eta' \\ 0& 1\end{pmatrix} \nabla_{\tilde z}, \qquad \Delta_z = \partial_{\tilde z_1}^2 + (1+\e^{2\alpha}|\eta' |^2)\partial_{\tilde z_2}^2 - 2\e^\alpha \eta' \partial_{\tilde z_1}\partial_{\tilde z_2} - \e^\alpha \eta'' \partial_{\tilde z_2}.\end{equation*}

We shall take our weight function $ \phi = \phi(\tilde z_2),$ and hence, 
$$\begin{aligned}
( \tilde u^\mathrm{app} + v)\cdot \nabla \phi+ \tilde \nu \Delta\phi &= \langle \e^\alpha \eta' \rangle ( \tilde u^\mathrm{app} + v) \cdot n \phi' + \tilde \nu (1+\e^{2\alpha }|\eta'|^2) \phi'' - \tilde \nu \e^\alpha \eta'' \phi'.
 \end{aligned}$$ 
Recall that $( \tilde u^\mathrm{app} + v) \cdot n = 0$ on $\{\tilde z_2 = 0\}$. By definition of $m(T)$, it follows that 
$$ |(\tilde u^\mathrm{app} + v)\cdot n| \le m(T) \tilde z_2 \quad\text{and}\quad \langle \e^\alpha \eta' \rangle ( \tilde u^\mathrm{app} + v) \cdot n \leq \langle \e^\alpha \eta' \rangle^2 m(T) \tilde z_2 .$$
Using this into the above identity, we get (if we prove later that $\phi'\geq0$)
$$\begin{aligned}
( \tilde u^\mathrm{app} + v) \cdot \nabla \phi+ \tilde \nu \Delta\phi &\le (1+ \e^{2\alpha }|\eta'|^2) \Big[ m(T) \tilde z_2 \phi' + \tilde \nu \phi'' \Big] - \tilde \nu \e^\alpha \eta'' \phi' . 
 \end{aligned}$$ 
We then choose the function $\phi$ so that $m(T) \tilde z_2 \phi' + \tilde \nu \phi'' = 0$, or equivalently, we can take 
$$\phi' (z)=Ce^{-\frac{m(T) \tilde z_{2}^2}{2\tilde\nu}},\qquad \phi (z)= C\sqrt{\frac{\tilde \nu}{m(T)}} \int_{0}^{ \sqrt{m(T) \over \tilde \nu} \tilde z_2} e^{- {s^2 \over 2}}\, ds,$$
for $C$ to be fixed. 
Hence, for $C=\sqrt{m(T)}$, we verify that $\phi'\geq0$ and we have
$$ ( \tilde u^\mathrm{app} + v)\cdot \nabla \phi+ \tilde \nu \Delta\phi \lesssim \tilde \nu \e^\alpha \sqrt{m(T)}.$$
Other properties of $\phi$ stated in the lemma follow directly. This proves the existence of the weight function $\phi$ as claimed. 
\end{proof}

With the weight function $\phi$ constructed in Lemma~\ref{lem-phi}, let us now estimate each term on the right of the identity \eqref{id-L2w}. First, by construction, we have 
\begin{equation*} \int\Big[ (\tilde u^\mathrm{app} + v) \cdot \nabla \phi+\tilde \nu \Delta\phi \Big] \frac{|\omega|^2}{2} \lesssim \tilde \nu \e^\alpha \|\omega \|_{L^2(\tilde \Omega^\epsilon)}^2.\end{equation*}
Next, for convenience, let us denote the weighted norm:
$$ \| v\|_{L^p_\phi(\tilde \Omega^\epsilon)}: = \|\phi^{1/p} v\|_{L^p(\tilde \Omega^\epsilon)}, \qquad p\ge 1.$$
Thanks to the estimate \eqref{L2-Rappw} and Lemma~\ref{lem-phi}, we have 
\begin{equation*}
\begin{aligned}
 \e \int \phi \omega\curl \tilde R^{\nu, \varepsilon}_\mathrm{app} &\lesssim \| \omega \|_{L^2_\phi} \| \phi\|_{L^\infty}^{1/2} \| \e \curl_z \tilde R^{\nu, \varepsilon}_\mathrm{app}\|_{L^2}
 \lesssim \tilde \nu^{1/4} ( \tilde\nu\varepsilon^2 + \e^{\alpha+\frac32}) \| \omega \|_{L^2_\phi}.
 \end{aligned}\end{equation*}
Next, using the boundary condition on $\omega$ \eqref{BC-resw}, the trace lemma~\ref{lem-trace2} and \eqref{Duapp_bord}, we estimate
\begin{equation*}\begin{aligned}
\frac {\tilde \nu} 2 \int_{\pa \tilde \Omega^\eps } \frac{\partial \phi}{\partial n} |\omega|^2 
&\lesssim \tilde \nu \|\omega\|_{L^2 (\pa \tilde \Omega^\eps )}^2
\lesssim \tilde \nu \Big[ \e^{2\alpha} \| v\cdot \tau \|_{L^2(\pa \tilde \Omega^\eps )}^2 + \e^{2\alpha-1} \Big]
 \\
 &\lesssim \tilde \nu \Big[ \e^{2\alpha} \| v \|_{L^2(\tilde \Omega^\epsilon)} \|\omega\|_{L^2(\tilde \Omega^\epsilon)} + \e^{2\alpha-1}\Big].
 \end{aligned}\end{equation*}
Note that here, the use of the refined trace lemma~\ref{lem-trace2} so that we have $\omega$ and not $\nabla v$ in the right
hand-side is crucial.

Finally, by recalling that $\nabla_{z} \omega^\mathrm{app}=\varepsilon^2(\nabla_{x}\curl_{x} u^\app )(\varepsilon\hat\tau,\varepsilon z)$ (see \eqref{curl-app}), the last integral can be estimated thanks to \eqref{all_bounds}:
\begin{equation*}
\Big| \int \phi \omega v \cdot \nabla \omega^\mathrm{app} \Big| \lesssim \e^2 \tilde\nu^{1/4} \| \omega \|_{L^2_\phi (\tilde \Omega^\epsilon)} \| v\|_{L^2(\tilde \Omega^\epsilon)}. \end{equation*}

Combining all the above estimates into \eqref{id-L2w}, we thus get the following estimate 
$$
\begin{aligned}
 \frac 12 \frac{d}{d\hat \tau} \| \omega \|_{L_\phi^2(\tilde \Omega^\epsilon)}^2 + \tilde \nu \| \nabla \omega \|^2_{L_\phi^2(\tilde \Omega^\epsilon)} 
&
 \lesssim \varepsilon^{\frac12} \| \omega \|_{L^2_\phi (\tilde \Omega^\epsilon)} \Big[\tilde\nu^{1/4}( \tilde\nu\varepsilon^{\frac32} + \e^{\alpha+1})+ \e^{\frac32} \tilde\nu^{1/4} \| v\|_{L^2(\tilde \Omega^\epsilon)} \Big]
 \\& \quad 
 + \tilde \nu \Big[ \e^\alpha \|\omega \|_{L^2(\tilde \Omega^\epsilon)}^2 +\e^{2\alpha} \| v \|_{L^2(\tilde \Omega^\epsilon)} \|\omega\|_{L^2(\tilde \Omega^\epsilon)} + \e^{2\alpha-1}\Big].
 \end{aligned}
 $$
Applying the Young's inequality, we obtain at once 
\begin{multline}\label{temp}
 \frac 12 \frac{d}{d\hat \tau} \| \omega \|_{L_\phi^2(\tilde \Omega^\epsilon)}^2 + \tilde \nu \| \nabla \omega \|^2_{L_\phi^2(\tilde \Omega^\epsilon)} 
 \lesssim \e \| \omega \|^2_{L^2_\phi (\tilde \Omega^\epsilon)} 
 + \e^3 \tilde \nu^{1/2} \| v \|^2_{L^2(\tilde \Omega^\epsilon)} \\
 +(\tilde \nu \varepsilon^\alpha+\tilde\nu^{3/2}\varepsilon^{4\alpha-3}) \|\omega\|_{L^2(\tilde \Omega^\epsilon)}^2 
 + \tilde\nu^{1/2}( \tilde\nu\varepsilon^{\frac32} + \e^{\alpha+1})^2+ \tilde \nu \e^{2 \alpha - 1}.
 \end{multline}
 To replace $\|\omega\|_{L^2(\tilde \Omega^\epsilon)}$ on the right by the corresponding weighted norm, we note from the proof of Lemma~\ref{lem-phi} that 
 \[
 \nabla_{z} \phi (z) \cdot n(z_{1})=\frac{(-\varepsilon^\alpha \eta')^2+1}{\langle \varepsilon^\alpha \eta'\rangle}\phi'(\tilde z_{2}) \approx \sqrt{m(T)}\quad \text{when}\quad\tilde z_{2}\leq \sqrt{\tilde \nu}.
 \]
 Hence, the properties of $\phi$ imply
\begin{equation}\label{omegaphi}
\begin{aligned}
 \|\omega\|_{L^2(\tilde \Omega^\epsilon)}^2&= \int_{\{\tilde z_2\geq \sqrt{\tilde \nu}\}} |\omega|^2 + \int_{\{\tilde z_2\leq \sqrt{\tilde \nu}\}} |\omega|^2 
 \\&\leq \tilde\nu^{-1/2} \|\omega\|_{L^2_{\phi}(\tilde \Omega^\epsilon)}^2+ \int_{\tilde \Omega^\epsilon} \nabla \phi \cdot n \omega^2\\
 &\leq \tilde\nu^{-1/2} \|\omega\|_{L^2_{\phi}(\tilde \Omega^\epsilon)}^2- \int_{\tilde \Omega^\epsilon} \phi (\partial_{z_{1}} n_{1}) \omega^2 - 2\int_{\tilde \Omega^\epsilon} \phi \omega n\cdot \nabla \omega\\
 &\lesssim \tilde\nu^{-1/2} \|\omega\|_{L^2_{\phi}(\tilde \Omega^\epsilon)}^2 +\tilde\nu^{1/2} \|\nabla \omega\|_{L^2_{\phi}(\tilde \Omega^\epsilon)}^2 ,
\end{aligned}
\end{equation}
because $|\partial_{z_{1}} n_{1}| \lesssim \varepsilon^\alpha$. The last term can then be absorbed into the left hand side of \eqref{temp}, when multiplied by $(\tilde \nu \varepsilon^\alpha+\tilde\nu^{3/2}\varepsilon^{4\alpha-3}) \ll \tilde \nu^{1/2}$. 
 
Finally, by anticipating the analysis of the next section, it is crucial that the constant terms in \eqref{temp} are smaller than $\tilde \nu^{1/2}\varepsilon^{2\alpha+1}$. This is possible, provided that $\tilde \nu\varepsilon^{2\alpha-1} \lesssim \tilde \nu^{1/2}\varepsilon^{2\alpha+1}$, or equivalently $\tilde\nu\lesssim \varepsilon^4$. Hence, if $\tilde\nu\lesssim \varepsilon^4$ and $m(T)\lesssim 1$, 
the estimate \eqref{temp} reduces to 
\begin{equation}\label{key-west}
 \frac{d}{d\hat \tau} \| \omega \|_{L_\phi^2(\tilde \Omega^\epsilon)}^2 + \tilde \nu \| \nabla \omega \|^2_{L_\phi^2(\tilde \Omega^\epsilon)} 
 \lesssim \e \| \omega \|^2_{L^2_\phi (\tilde \Omega^\epsilon)} 
 + \e^3 \tilde \nu^{1/2} \| v \|^2_{L^2(\tilde \Omega^\epsilon)} 
 + \tilde\nu^{1/2} \e^{2\alpha+2}+ \tilde \nu \e^{2 \alpha - 1},
 \end{equation}
 where we have assumed $\varepsilon$ small enough.

\subsubsection{End of the proof of Proposition~\ref{prop-vort}.}
In this section, we shall close the vorticity estimate, assuming $\tilde\nu\lesssim \varepsilon^4$ and $m(T)\lesssim 1$ (which is implied by the assumptions of Proposition~\ref{prop-vort}, see Remark~\ref{rem-mT}). Precisely, let us introduce 
\begin{equation*}
\begin{aligned}
\cN(\hat \tau): &= 
 \tilde \nu^{1/2} \eps^2 \| v (\hat \tau)\|^2_{L^2(\tilde \Omega^\epsilon)} + \| \omega (\hat \tau )\|^2_{L_\phi^2(\tilde \Omega^\epsilon)}.
\end{aligned}\end{equation*}

Adding the assumption $\tilde\nu\lesssim \varepsilon^4$, the velocity estimate proved in Proposition~\ref{prop-L2} reads: 
$$
\frac12 \frac{d}{d\hat \tau} \| v\|_{L^2(\tilde \Omega^\epsilon)}^2 + \tilde \nu \| \nabla v \|_{L^2(\tilde \Omega^{\e})}^2 \lesssim \e \| v\|_{L^2(\tilde \Omega^\epsilon)}^2 + \| v\|_{L^2(\tilde \Omega^\epsilon)} \| \omega\|_{L^2(\tilde \Omega^\epsilon)}
+ \e^{2\alpha +2} $$
so, using \eqref{omegaphi}, we get
$$\begin{aligned}
 &\frac{\tilde \nu^{1/2} \eps^2}2 \frac{d}{d\hat \tau} \| v\|_{L^2(\tilde \Omega^\epsilon)}^2 + \tilde \nu^{3/2} \eps^2 \| \nabla v \|_{L^2(\tilde \Omega^{\e})}^2 
 \\&\lesssim \tilde \nu^{1/2} \eps^3 \| v\|_{L^2(\tilde \Omega^\epsilon)}^2 + \tilde \nu^{1/2} \eps \| \omega\|_{L^2(\tilde \Omega^\epsilon)}^2
+ \tilde \nu^{1/2} \e^{2\alpha +4}
\\
& \lesssim \tilde \nu^{1/2} \eps^3 \| v\|_{L^2(\tilde \Omega^\epsilon)}^2 + 
\eps \|\omega\|_{L^2_{\phi}(\tilde \Omega^\epsilon)}^2 +\tilde\nu \eps \|\nabla \omega\|_{L^2_{\phi}(\tilde \Omega^\epsilon)}^2
+ \tilde \nu^{1/2} \e^{2\alpha +4} .
\end{aligned}$$
Together with the vorticity estimate \eqref{key-west}, this implies (for $\varepsilon$ small enough)
$$
 \frac{d}{d\hat \tau} \cN(\hat \tau) + \tilde \nu^{3/2} \eps^2 \| \nabla v \|_{L^2(\tilde \Omega^{\e})}^2 +\tilde \nu \| \nabla \omega \|^2_{L_\phi^2(\tilde \Omega^\epsilon)} \lesssim \e \cN(\hat \tau) + \tilde\nu^{1/2} \e^{2\alpha+2}+ \tilde \nu \e^{2 \alpha - 1} .
 $$
 By the Gronwall's inequality, this yields that
\[
\cN(\hat \tau) + \int_{0}^{\hat \tau} e^{\varepsilon C(\hat \tau-s)} \Big(\tilde \nu^{3/2} \eps^2 \| \nabla v \|_{L^2(\tilde \Omega^{\e})}^2 +\tilde \nu \| \nabla \omega \|^2_{L_\phi^2(\tilde \Omega^\epsilon)}\Big)\, ds \leq ( \tilde\nu^{1/2} \e^{2\alpha+1}+ \tilde \nu \e^{2 \alpha - 2} )e^{\varepsilon C\hat \tau}
\]
for any $\hat \tau\leq T/\varepsilon$. We deduce at once that
\begin{gather*}
 \sup_{0\le \hat \tau \leq T/\varepsilon} \e^2 \| v (\hat \tau)\|_{L^2(\tilde \Omega^\epsilon)}^2
 \lesssim \e^{2\alpha+1}+ \tilde \nu^{1/2} \e^{2 \alpha - 2} \\
 \sup_{0\le \hat \tau \leq T/\varepsilon} \| \omega (\hat \tau )\|^2_{L_\phi^2(\tilde \Omega^\epsilon)}
 \lesssim \tilde\nu^{1/2} \e^{2\alpha+1}+ \tilde \nu \e^{2 \alpha - 2} 
\\ \tilde \nu \eps^2 \| \nabla v \|_{L^2((0,T/\varepsilon)\times \tilde \Omega^{\e})}^2 +\tilde \nu^{1/2} \| \nabla \omega \|^2_{L^2(0,T/\varepsilon;L_\phi^2(\tilde \Omega^\epsilon)} 
 \lesssim \e^{2\alpha+1}+ \tilde \nu^{1/2} \e^{2 \alpha - 2}.
\end{gather*}
Moreover, by Lemma~\ref{lem-phi}, $\tilde\nu^{1/2}\| \omega (\hat \tau )\|^2_{L^2(\{z_2 - \e^\alpha \eta(z_1)\gtrsim \sqrt{\tilde \nu}\})} \leq \| \omega (\hat \tau )\|^2_{L_\phi^2(\tilde \Omega^\epsilon)}$, which gives the second estimate of Proposition~\ref{prop-vort}.

Finally, we use again \eqref{omegaphi} to write
\begin{align*}
 \|\omega\|_{L^2(0, T/\varepsilon; L^2(\tilde \Omega^\epsilon))}^2 
 &\lesssim
 \tilde\nu^{-1/2} \frac{T}{\varepsilon} \|\omega\|_{L^\infty(0,T/\varepsilon;L^2_{\phi}(\tilde \Omega^\epsilon))}^2 +\tilde\nu^{1/2} \|\nabla \omega\|_{L^2(0, T/\varepsilon;L^2_{\phi}(\tilde \Omega^\epsilon))}^2\\
 &\lesssim \varepsilon^{-1} \e^{2\alpha+1}+ \tilde \nu^{1/2} \e^{2 \alpha - 3} 
\end{align*}
which yields the last estimates of Proposition~\ref{prop-vort}. The proof of the proposition is complete.

\subsection{$L^\infty$ estimates} 

To go from Proposition~\ref{prop-vort} to the final stability result, we need to show that Assumption \eqref{assmp-mT} is satisfied. The estimate of $\| \nabla v \|_{L^\infty}$ will come from an elliptic estimate in $\tilde \Omega^\varepsilon$ (see Proposition~\ref{prop-omegav} in Appendix~\ref{appdx3}) where it will be important that $\| \omega \|_{L^\infty}+ \varepsilon^\alpha \| v \|_{L^\infty} \lesssim 1$. Deriving this uniform estimate is exactly the purpose of this section.

\begin{proposition}\label{prop-infty}
 Assume that $\tilde \nu \lesssim \varepsilon^6$. If there exists an absolute constant $K$ such that 
\begin{equation} \label{major_hyp}
\sup_{0\le \hat \tau\le T/\varepsilon}\|\nabla_z v (\hat \tau)\|_{L^\infty(\tilde \Omega^\epsilon)}
+ \sup_{0\le \hat \tau\le T/\varepsilon} \tilde \nu^{K } \| v (\hat \tau)\|_{H^2(\tilde \Omega^\epsilon)} \leq 1\quad \text{with} \quad T\leq T_{0},
\end{equation}
then, we have the uniform bound (for $\varepsilon$ small enough)
$$
\sup_{0\le \hat \tau\le T/\varepsilon} (\| \omega(\hat \tau)\|_{L^\infty(\tilde \Omega^\epsilon)} + \|v(\hat \tau)\|_{L^\infty(\tilde \Omega^\epsilon)}) \lesssim \e^{\alpha}.
$$
\end{proposition}

\begin{remark}\label{rem-enu}
We can now explain where the condition between $\tilde\nu$ and $\varepsilon$ stated in \eqref{nu-eps} (and in the main theorem) is used in our analysis:
\begin{itemize}
 \item Concerning the uniform bound of the velocity, we will use another elliptic estimate in $\tilde \Omega^\varepsilon$, namely Proposition~\ref{vinfty-vH2}, where we need $\varepsilon^{1/2} \| v \|_{L^2} \lesssim 1$. Such an estimate comes from Proposition~\ref{prop-vort} if and only if $\tilde\nu^{1/4}\varepsilon^{\alpha-1}\lesssim \varepsilon^{1/2}$ i.e. $\tilde\nu \lesssim \varepsilon^{6-4\alpha}\leq \varepsilon^6$ (for any $\alpha>0$, arbitrary small).
 \item Looking again at Proposition~\ref{vinfty-vH2}, using that $\|v\|_{H^2} \lesssim \tilde\nu^{-K}$, we will need the bound $C\varepsilon^\alpha \ln(2+ \tilde\nu^{-K} \varepsilon^{-2} )\leq 1/2$ to hold for any fixed $C$ and $\eps$ small enough, i.e. $\varepsilon^\alpha \ll 1/\ln \tilde \nu$. The condition $\tilde \nu \geq \varepsilon^{N_{1}}$ will guarantee this for any $N_{1}$ and any $\alpha=1/N_{0}$ (with $N_{0}, N_{1} \in \N$ arbitrary large).
\end{itemize}
\end{remark}

We now turn to the proof of Proposition~\ref{prop-infty}. As said in the previous remark, if $\tilde \nu \lesssim \varepsilon^6$, Proposition~\ref{prop-vort} gives that 
\[
\varepsilon^{\frac12} \| v \|_{L^2(\tilde \Omega^\varepsilon)} \lesssim \varepsilon^\alpha \quad \text{and} \quad \| \omega \|_{L^2(\{z_2 - \e^\alpha \eta(z_1)\gtrsim \sqrt{\tilde \nu}\})} \lesssim \e^{\alpha+\frac12}.
\]
So using Proposition~\ref{vinfty-vH2} and the assumption $\| v (\hat \tau)\|_{H^2(\tilde \Omega^\epsilon)} \leq \tilde \nu^{-K}\leq \varepsilon^{-KN_{1}}$ (with $N_{1}$ defined in \eqref{nu-eps}), wet get for small enough $\eps$: 
 \begin{equation}\label{pfff}
 \begin{aligned}
\|v\|_{L^\infty(\tilde{\Omega}^{\eps})}
&\lesssim \eps^\alpha + \| \omega \|_{L^2(\tilde{\Omega}^{\eps})} 
 \Big( \ln( 2 + { 1 \over \eps^{2+KN_{1}}})\Big)^{1\over 2}\\
 & \lesssim \eps^{\alpha}
+ \| \omega \|_{L^2{(\{ z_2 - \e^\alpha \eta(z_1) \lesssim \sqrt {\tilde \nu}\})}} |\ln \eps|^{1\over 2}
 + \| \omega \|_{L^2{(\{z_2 - \e^\alpha \eta(z_1) \gtrsim \sqrt {\tilde \nu}\})}} |\ln \eps|^{1\over 2}
 \\
 &\lesssim \eps^\alpha + 
\tilde \nu^{1/4} \e^{-1/2} |\ln \eps|^{1\over 2}\| \omega\|_{L^\infty}
 + \eps^{ \alpha+\frac12} |\ln \eps|^{1\over 2}\\
 &\lesssim \eps^\alpha + \varepsilon^{\frac12} \| \omega\|_{L^\infty(\tilde{\Omega}^{\eps})},
\end{aligned}\end{equation} 
where we have also used $\tilde \nu \lesssim \varepsilon^6$.

Next, we recall that the vorticity equation reads
\begin{equation}
\label{equationvort}
\left\{ \begin{aligned}
&\partial_{\hat\tau}\omega + (\tilde u^\mathrm{app} + v)\cdot \nabla_z \omega - \tilde \nu \Delta \omega = - v\cdot \nabla_z \omega^\mathrm{app} + \e \curl_z \tilde R^{\nu, \varepsilon}_\mathrm{app}
\\
&\omega = \omega_\Gamma:= 2 D_z (\tilde u^\app) n \cdot \tau + \varepsilon \lambda \tilde u^\app \cdot \tau + (2 \kappa+\varepsilon\lambda) v\cdot \tau , \qquad \forall~z\in \partial \tilde \Omega^\varepsilon.
\end{aligned} \right.
\end{equation}
We can then use a classical comparison principle for solutions of parabolic equations. Applying \cite[Proposition 21, page 37]{AW} to the functions 
$$ (\hat \tau, z) \rightarrow \omega(\hat \tau, z) \quad \text{and} \quad (\hat \tau,z) \rightarrow \| \omega_{\Gamma} \|_{L^\infty_{\hat \tau}L^\infty} + \int_0^{\hat \tau} \Big( \| v\cdot \nabla_z \omega^\mathrm{app} \|_{L^\infty}+ \e \| \curl_z \tilde R^{\nu, \varepsilon}_\mathrm{app} \|_{L^\infty} \Big)\; ds $$
we obtain over $(0,T/\eps)$:
$$ 
\| \omega \|_{L^\infty_{\hat \tau} L^\infty} \le \| \omega_{\Gamma} \|_{L^\infty_{\hat \tau} L^\infty} + \int_0^{T/\varepsilon} \Big( \| v\cdot \nabla_z \omega^\mathrm{app} \|_{L^\infty}+ \e \| \curl_z \tilde R^{\nu, \varepsilon}_\mathrm{app} \|_{L^\infty} \Big)\; ds .
$$
We recall that $\omega\vert_{t=0} = 0$ in the above. Functions $\tilde R^{\nu, \varepsilon}_\mathrm{app}$ and $\omega^\mathrm{app}$ are defined in \eqref{def-vRapp}-\eqref{reNSv-eqs} and \eqref{curl-app}, respectively. Hence, in view of the uniform bounds from \eqref{all_bounds} on the approximate solution, we obtain 
\begin{gather*}
\| \nabla_z \omega^\app \|_{L^\infty}=\varepsilon^2 \| \nabla_x \curl_{x} u^\app \|_{L^\infty(\Omega^\varepsilon)} \lesssim \e^2, \\
 \e \| \curl_z \tilde R^{\nu, \varepsilon}_\mathrm{app} \|_{L^\infty} \leq \varepsilon^3 \tilde \nu \| \curl_x \Delta u^\mathrm{app}\|_{L^\infty} +\varepsilon^2\|\curl_{x} R^\mathrm{app}\|_{L^\infty} \lesssim \varepsilon^3\tilde \nu + \e^{\alpha+2}, 
\end{gather*}
and the boundary estimate 
$$\|\omega_\Gamma\|_{L^\infty} \lesssim \varepsilon^\alpha \| v\|_{L^\infty} + \varepsilon \| D_{x} u^\app\|_{L^\infty} + \varepsilon^\alpha \| u^\app\|_{L^\infty} \lesssim \e^\alpha \| v\|_{L^\infty} + \e^{\alpha},$$
where we have used \eqref{kappa} and \eqref{all_bounds} (as $v$ and $u^\app$ are continuous).
By using that $T \leq T_{0}$, the above yields 
\begin{equation*}
\| \omega \|_{L^\infty_{\hat \tau} L^\infty} 
\le \e^\alpha \| v\|_{L^\infty_{\hat \tau}L^\infty} + \e^{\alpha}.
\end{equation*}

Then considering \eqref{pfff}, we get that 
$$
 \|\omega \|_{L^\infty_{\hat \tau}L^\infty} \lesssim \e^{\alpha}
$$
provided that $\varepsilon$ is small enough. Finally, by plugging this estimate into \eqref{pfff}, we also obtain that
 $$ \|v\|_{L^\infty_{\hat \tau}L^\infty} \lesssim \eps^\alpha.$$ 
This ends the proof of Proposition~\ref{prop-infty}.

\subsection{Higher order estimates}

In view of the quantities appearing in Assumption \eqref{major_hyp}, we derive in this section some higher order estimates.

\begin{proposition}\label{prop-higher}
 Assume that $\tilde \nu \lesssim \varepsilon^6$. If there exists an absolute constant $K > 0$ such that 
\begin{equation*}
\sup_{0\le \hat \tau\le T/\varepsilon}\|\nabla_z v (\hat \tau)\|_{L^\infty(\tilde \Omega^\epsilon)}
+ \sup_{0\le \hat \tau\le T/\varepsilon} \tilde \nu^{K } \| v (\hat \tau)\|_{H^2(\tilde \Omega^\epsilon)} \leq 1\quad \text{with} \quad T\leq T_{0},
\end{equation*}
then, we have the uniform bound (for $\varepsilon$ small enough)
 \begin{gather}
 \label{dtvL2L2}
 \tilde \nu^{1 \over 2 } \| \partial_{\hat \tau}v \|_{L^2(0,T/\varepsilon ;L^2 (\tilde \Omega^\varepsilon))} + \tilde \nu \| \nabla v \|_{L^\infty(0,T/\varepsilon ;L^2 (\tilde \Omega^\varepsilon))} \lesssim \e^{\alpha- {1 \over 2}}, \\
 \label{dtvLinftyL2} \tilde \nu^{1 \over 2 } \| \partial_{\hat \tau} v \|_{L^\infty(0,T/\varepsilon ;L^2 (\tilde \Omega^\varepsilon))} + \tilde \nu \| \nabla \partial_{\hat \tau} v \|_{L^2(0,T/\varepsilon ;L^2 (\tilde \Omega^\varepsilon))}
 \lesssim \e^{\alpha - {1 \over 2} }, \\
 \label{dtvLinftyH1}
 \tilde {\nu}^{3 \over 2} \| \nabla \partial_{\hat \tau} v \|_{L^\infty(0,T/\varepsilon ;L^2 (\tilde \Omega^\varepsilon))} \lesssim \e^{\alpha - {1 \over 2}}, \\
 \label{D2omega}
 \tilde {\nu}^{5\over 2 } \| \omega \|_{L^\infty(0,T/\varepsilon ;H^2 (\tilde \Omega^\varepsilon))} \lesssim \e^{ \alpha - {1\over 2}}.
 \end{gather}
\end{proposition}

In the sequel of this section, we prove this proposition. First, we use the boundary conditions \eqref{partialvn} and $(\partial_{\hat \tau} v)\cdot n =v\cdot n= 0$ to write
 \begin{align*}
 \int_{\tilde \Omega^\e} \Delta v \cdot \partial_{\hat \tau}v \, dx =& - {1 \over 2} {d \over d\hat \tau} \| \nabla v \|_{L^2(\tilde \Omega^{\e})}^2
 - \int_{\pa \tilde \Omega^\eps} \partial_{n} v \cdot \partial_{\hat \tau}v
 \\=& - {1 \over 2} {d \over d\hat \tau} \| \nabla v \|_{L^2(\tilde \Omega^{\e})}^2
 - \int_{\pa \tilde \Omega^\eps} (\partial_{n} v \cdot \tau) (\partial_{\hat \tau}v \cdot \tau)\\
 =&
 {d \over d\hat \tau} \Big( -{1 \over 2} \| \nabla v \|_{L^2(\tilde \Omega^{\e})}^2 + {1 \over 2}\int_{\pa \tilde \Omega^\eps} (\kappa+\varepsilon\lambda) |v\cdot \tau|^2 
 +\int_{\pa \tilde \Omega^\eps} (2D_z (\tilde u^\app) n \cdot \tau+\varepsilon\lambda \tilde u^\app \cdot \tau )(v\cdot \tau) \Big)\\
& - \int_{\pa \tilde \Omega^\eps} \partial_{\hat \tau} (2 D_z (\tilde u^\app) n \cdot \tau+\varepsilon\lambda \tilde u^\app \cdot \tau)( v \cdot \tau)\\
 =&
 {d \over d\hat \tau} \Big( -{1 \over 2} \| \nabla v \|_{L^2(\tilde \Omega^{\e})}^2 + {1 \over 2}\int_{\pa \tilde \Omega^\eps} (\kappa+\varepsilon\lambda) |v |^2 
 +\int_{\pa \tilde \Omega^\eps} (2 D_z (\tilde u^\app) n +\varepsilon\lambda \tilde u^\app ) \cdot v \Big)
 \\&- \int_{\pa \tilde \Omega^\eps} \partial_{\hat \tau}(2 D_z (\tilde u^\app) n+\varepsilon\lambda \tilde u^\app ) \cdot v .
 \end{align*}
Then, taking the scalar product of \eqref{prob-resv} with $\partial_{\hat \tau} v$, we obtain that 
 \begin{align*}
 \tilde \nu{d \over d\hat \tau} \Big( {1 \over 2 } \| \nabla v \|_{L^2}^2
 - { 1 \over 2 } \int_{\pa \tilde \Omega^\eps} &(\kappa + \varepsilon \lambda) |v|^2 
 - \int_{\pa \tilde \Omega^\eps} (2 D_z (\tilde u^\app) n +\varepsilon\lambda \tilde u^\app ) \cdot v \Big)
 + \| \partial_{\hat \tau} v \|_{L^2}^2 \\
 \leq&
 \| \partial_{\hat \tau} v\|_{L^2}\Big( \| \nabla v \|_{L^2} \| \tilde u^\app +v \|_{L^\infty} + \|v\|_{L^2} \| \nabla_{z} \tilde u^\app \|_{L^\infty} + \| \varepsilon \tilde R^{\nu,\varepsilon}_{\app} \|_{L^2}\Big)\\
& + \tilde \nu \Big|\int_{\pa \tilde \Omega^\eps} \partial_{\hat \tau} (2 D_z (\tilde u^\app) n+\varepsilon\lambda \tilde u^\app ) \cdot v \Big| 
 \end{align*}
 By using the $L^\infty$ bound on $u^\app$ and $v$ (see \eqref{all_bounds} and Proposition~\ref{prop-infty}), that $\| \nabla_{z} \tilde u^\app \|_{L^\infty}=\varepsilon \| \nabla_{x} u^\app \|_{L^\infty}\lesssim \varepsilon^\alpha$ (see again \eqref{all_bounds}), the estimate of the $L^2$ norm of $\varepsilon \tilde R^{\nu,\varepsilon}_{\app}$ \eqref{Renuapp-L2}, and that $ \| \partial_{\hat \tau} D_z (\tilde u^\app)\|_{L^\infty}= \varepsilon^2 \| \partial_{t} D_x (u^\app)\|_{L^\infty}$ is of the same order as $ \varepsilon^2\|D_x (u^\app)\|_{L^\infty}$ (see Remark~\ref{rem-dtau}) i.e. less than $C\varepsilon^{\alpha+1}$, the Young inequality yields (for $\varepsilon$ small enough)
\begin{multline*}
 \tilde \nu {d \over d\hat \tau} \Big( {1 \over 2 } \| \nabla v \|_{L^2}^2
 - { 1 \over 2 } \int_{\pa \tilde \Omega^\eps} (\kappa+\varepsilon\lambda) |v|^2 
 - \int_{\pa \tilde \Omega^\eps} (2 D_z (\tilde u^\app) n +\varepsilon\lambda \tilde u^\app ) \cdot v \Big)
 + {1 \over 2 } \| \partial_{\hat \tau} v \|_{L^2}^2 \\
 \lesssim \| \nabla v \|_{L^2}^2 + \varepsilon^{2\alpha} \|v\|_{L^2}^2 + \e^{ 2 \alpha + 3}
 + \tilde \nu \frac1\varepsilon \varepsilon^{\alpha+1} \varepsilon^\alpha.
 \end{multline*}
Consequently, by integrating in time between $0$ and $t$, we get that for any $t\in [0,T/\varepsilon]$ we have
 \begin{align*}
 { \tilde \nu \over 2 } \| \nabla v \|_{L^2}^2+ {1 \over 2 } \| \partial_{\hat \tau} v \|_{L^2(0,t;L^2)}^2
\lesssim& { \tilde \nu \over 2 } \int_{\pa \tilde \Omega^\eps} |\kappa+\varepsilon\lambda| |v|^2 + \tilde \nu \int_{\pa \tilde \Omega^\eps} |(2 D_z (\tilde u^\app) n +\varepsilon\lambda \tilde u^\app ) \cdot v | \\
& + \| \nabla v \|_{L^2(0,t;\Omega^{\e})}^2 + \varepsilon^{2\alpha-1} T\|v\|_{L^\infty L^2}^2 + \e^{ 2 \alpha + 2}T,
 \end{align*}
 where we have used that $\tilde \nu \varepsilon^{-1}\lesssim \varepsilon^3$. Using that $\| \kappa +\varepsilon\lambda \|_{L^\infty} \lesssim \eps^\alpha$ and $\| v\|_{L^\infty} + \| D_z (\tilde u^\app)\|_{L^\infty} + \varepsilon \|\lambda \tilde u^\app\|_{L^\infty} \lesssim \varepsilon^\alpha$ ({\it cf} Proposition~\ref{prop-infty}), we get
 $$ \tilde \nu \|\nabla v \|_{L^\infty(0,T/\varepsilon;L^2)}^2 + \ \| \partial_{\hat \tau}v \|_{L^2(0,T/\varepsilon;L^2)}^2
 \lesssim 
 \tilde \nu \varepsilon^{2\alpha-1}+ \tilde\nu^{-1} \varepsilon^{2\alpha-1}+ \e^{4 \alpha - 2} + \e^{ 2 \alpha + 2}T$$
 and hence
 $$ \tilde \nu^2 \|\nabla v \|_{L^\infty(0,T/\varepsilon;L^2)}^2 + \tilde \nu \| \partial_{\hat \tau}v \|_{L^2(0,T/\varepsilon;L^2)}^2
 \lesssim \e^{2 \alpha - 1},$$
 which yields \eqref{dtvL2L2}.

 We shall now estimate $\partial_{\hat \tau}v$ uniformly in time. By taking the time derivative of 
 \eqref{prob-resv} and \eqref{partialvn}, we find that 
\begin{equation}
\label{eqdtauv}
\begin{aligned}
\partial_{\hat \tau}( \partial_{\hat \tau}v) + (\tilde u^\mathrm{app} + v)\cdot \nabla_z \partial_{\hat \tau}v + \partial_{\hat \tau}v\cdot \nabla_z \tilde u^\mathrm{app} + \nabla_z \partial_{ \hat \tau}p 
&=\tilde \nu \partial_{\hat \tau}\Delta_z v + F, 
\\ \Div \partial_{\hat \tau}v&= 0,
\end{aligned}\end{equation}
where
$$ F= 
 \epsilon \partial_{\hat \tau}\tilde R^{\nu, \varepsilon}_\mathrm{app}
- \partial_{\hat \tau}( \tilde u^\mathrm{app} + v)\cdot \nabla_z v - v \cdot \nabla_z \partial_{\hat \tau}\tilde u^\mathrm{app}
$$
and the boundary conditions
$$
\partial_{\hat \tau}v \cdot n = 0
, \, \partial_{n} \partial_{\hat \tau}v \cdot \tau= -2 D_z (\partial_{\hat \tau} \tilde u^\app) n \cdot \tau - \varepsilon \lambda \partial_{\hat \tau} \tilde u^\app \cdot \tau -(\kappa+\varepsilon\lambda) \partial_{\hat \tau} v\cdot \tau 
. $$
By the trace Lemma~\ref{lem-trace}, we write that
\[
 \Big| \int_{\pa \tilde \Omega^\eps} \partial_{n} \partial_{\hat \tau}v \cdot \partial_{\hat \tau}v \Big| = \Big| \int_{\pa \tilde \Omega^\eps} (\partial_{n} \partial_{\hat \tau}v \cdot \tau )(\partial_{\hat \tau}v \cdot \tau)\Big| 
\lesssim \varepsilon^{\alpha+2} + \varepsilon^\alpha \| \partial_{\hat \tau}v \|_{L^2(\pa \tilde \Omega^\eps)}^2
\lesssim \varepsilon^{\alpha+2} + \varepsilon^\alpha \| \partial_{\hat \tau}v \|_{L^2}\| \partial_{\hat \tau} \nabla v \|_{L^2}
\]
where we have used that $ \| \partial_{\hat \tau} D_z (\tilde u^\app)\|_{L^\infty}= \varepsilon^2 \| \partial_{t} D_x (u^\app)\|_{L^\infty}\lesssim \varepsilon^{\alpha+1}$, that $ \| \partial_{\hat \tau} \tilde u^\app\|_{L^\infty}= \varepsilon \| \partial_{t} u^\app\|_{L^\infty}\lesssim \varepsilon$ and $|\kappa| + \varepsilon |\lambda|\lesssim \varepsilon^\alpha$.
Then, taking the scalar product of \eqref{eqdtauv} with $\partial_{\hat \tau}v$, the standard energy estimate gives
 $$ {d \over d\hat \tau } {1 \over 2} \| \partial_{\hat \tau} v \|_{L^2}^2
 + \frac{\tilde \nu}2 \| \nabla \partial_{\hat \tau}v \|_{L^2}^2 \lesssim
 \| \partial_{\hat \tau}v \|_{L^2}^2 + 1 + \varepsilon^{2\alpha+2} \|v\|_{L^2}^2,$$
 where we have used that $\|\nabla_z \tilde u^\mathrm{app}\|_{L^\infty}\lesssim \varepsilon^\alpha$, that $ \| \epsilon \partial_{\hat \tau}\tilde R^{\nu, \varepsilon}_\mathrm{app}\|_{L^2}=\varepsilon^2\| \partial_{t}R^{\nu, \varepsilon}_\mathrm{app}\|_{L^2}\lesssim \varepsilon^{\alpha+\frac52}$ (combining \eqref{Renuapp-L2} with Remark~\ref{rem-dtau}), that $\|\partial_{\hat \tau} \tilde u^\mathrm{app} \|_{L^2(\tilde \Omega^\varepsilon)}=\|\partial_{t} u^\mathrm{app} \|_{L^2(\Omega^\varepsilon)}\lesssim 1$ (see Remark~\ref{rem-dtau}), that $\| \nabla v \|_{L^\infty} \lesssim 1$ by assumption, and again that $\| \partial_{\hat \tau} D_z (\tilde u^\app)\|_{L^\infty} + \varepsilon\| \lambda \partial_{\hat \tau} \tilde u^\app\|_{L^\infty} \lesssim \varepsilon^{\alpha+1}$. We know by Proposition~\ref{prop-vort} that $\varepsilon^{2\alpha+2} \|v\|_{L^2}^2\lesssim \varepsilon^{3\alpha+1}$. Integrating with respect to time, we infer that
 \[
\| \partial_{\hat \tau} v \|_{L^\infty(0,T/\varepsilon ;L^2)}^2 + \tilde \nu \| \nabla \partial_{\hat \tau} v \|_{L^2(0,T/\varepsilon ;L^2)}^2
\lesssim \| \partial_{\hat \tau}v \|_{L^2(0,T/\varepsilon ;L^2)}^2 + \varepsilon^{-1}T.
 \]
 Therefore, \eqref{dtvL2L2} implies \eqref{dtvLinftyL2}.

Next, to prove \eqref{dtvLinftyH1}, we can take the scalar product of \eqref{eqdtauv} by $\partial_{\hat \tau\hat \tau}v$, and as in the beginning of this proof, we obtain that
\begin{align*}
 \tilde \nu{d \over d\hat \tau} \Big( {1 \over 2 } \| \nabla \partial_{\hat \tau}v \|_{L^2}^2 
& - { 1 \over 2 } \int_{\pa \tilde \Omega^\eps} (\kappa+\varepsilon\lambda) |\partial_{\hat \tau} v|^2 
 - \int_{\pa \tilde \Omega^\eps} (2 D_z (\partial_{\hat \tau} \tilde u^\app) n +\varepsilon\lambda \partial_{\hat \tau} \tilde u^\app ) \cdot \partial_{\hat \tau}v \Big)
 + \| \partial_{\hat \tau} \partial_{\hat \tau} v \|_{L^2}^2 \\
 \leq &
 \| \partial_{\hat \tau} \partial_{\hat \tau} v\|_{L^2}\Big( \| \nabla \partial_{\hat \tau} v \|_{L^2} \| \tilde u^\app +v \|_{L^\infty} + \|\partial_{\hat \tau} v\|_{L^2} \| \nabla_{z} \tilde u^\app \|_{L^\infty} + \| F \|_{L^2}\Big) \\
& + \tilde \nu \Big|\int_{\pa \tilde \Omega^\eps} \partial_{\hat \tau}(2 D_z (\partial_{\hat \tau} \tilde u^\app) n +\varepsilon\lambda \partial_{\hat \tau} \tilde u^\app ) \cdot \pa_{\hat \tau}v \Big| 
 \end{align*}
hence
\begin{multline*}
 \tilde \nu{d \over d\hat \tau} \Big( {1 \over 2 } \| \nabla \partial_{\hat \tau}v \|_{L^2}^2
 - { 1 \over 2 } \int_{\pa \tilde \Omega^\eps} (\kappa+\varepsilon\lambda) |\partial_{\hat \tau} v|^2 
 - \int_{\pa \tilde \Omega^\eps} (2 D_z (\partial_{\hat \tau} \tilde u^\app) n +\varepsilon\lambda \partial_{\hat \tau} \tilde u^\app ) \cdot \partial_{\hat \tau}v \Big)
 + \frac12 \| \partial_{\hat \tau} \partial_{\hat \tau} v \|_{L^2}^2 \\
 \leq
\| \nabla \partial_{\hat \tau} v \|_{L^2}^2 +\varepsilon^{2\alpha} \|\partial_{\hat \tau} v\|_{L^2}^2 +1,
 \end{multline*}
 where we have used the same estimates as in the proof of \eqref{dtvLinftyL2}.
 Consequently, we can integrate in time between $0$ and $t$ (for $t\leq T/\varepsilon$):
 \begin{align*}
 { \tilde \nu \over 2 } \| \nabla \partial_{\hat \tau} v \|_{L^2}^2+ {1 \over 2 } \| \partial_{\hat \tau} \partial_{\hat \tau} v \|_{L^2(0,t;L^2)}^2
\: \lesssim \: & { \tilde \nu \over 2 } \int_{\pa \tilde \Omega^\eps} |\kappa+\varepsilon\lambda| |\partial_{\hat \tau} v|^2 + \tilde \nu \int_{\pa \tilde \Omega^\eps} |(2 D_z (\partial_{\hat \tau} \tilde u^\app) n +\varepsilon\lambda \partial_{\hat \tau} \tilde u^\app ) \cdot \partial_{\hat \tau}v| \\
\: & + \| \nabla \partial_{\hat \tau} v \|_{L^2(0,t;\Omega^{\e})}^2 + \varepsilon^{2\alpha-1} T\|\partial_{\hat \tau} v\|_{L^\infty L^2}^2 + \varepsilon^{-1}T,
 \end{align*}
We know that $\| D_z (\partial_{\hat \tau}\tilde u^\app) \|_{L^\infty} +\varepsilon \|\lambda \partial_{\hat \tau} \tilde u^\app\|_{L^\infty} \lesssim \varepsilon^{\alpha+1}$, hence the trace Lemma~\ref{lem-trace} implies that
 \begin{align*}
 { \tilde \nu \over 2 } \| \nabla \partial_{\hat \tau} v \|_{L^2}^2+ {1 \over 2 } \| \partial_{\hat \tau} \partial_{\hat \tau} v \|_{L^2(0,t;L^2)}^2
\: \lesssim \: & \tilde \nu \varepsilon^\alpha \| \partial_{\hat \tau} v \|_{L^\infty(0,t;L^2)} \| \nabla \partial_{\hat \tau} v \|_{L^\infty(0,t;L^2)} + \tilde \nu \varepsilon^{\alpha+1}\\
& + \| \nabla \partial_{\hat \tau} v \|_{L^2(0,t;\Omega^{\e})}^2 + \varepsilon^{2\alpha-1} T\|\partial_{\hat \tau} v\|_{L^\infty L^2}^2 + \varepsilon^{-1}T,
 \end{align*}
so \eqref{dtvL2L2} and \eqref{dtvLinftyL2} allow us to state that
 $$ \tilde \nu \|\nabla \partial_{\hat \tau} v\|_{L^\infty(0,T/\varepsilon;L^2)}^2 \: \lesssim \: \tilde \nu^{-2}\varepsilon^{2\alpha-1}
 + \tilde \nu^{-1} \varepsilon^{4\alpha-2} \lesssim \tilde \nu^{-2}\varepsilon^{2\alpha-1}
 ,$$
which yields \eqref{dtvLinftyH1}:
 $$
 \tilde \nu^{1 \over 2 }\|\nabla \partial_{\hat \tau} v\|_{L^\infty(0,T/\varepsilon;L^2)} \lesssim \tilde \nu^{-1} \e^{\alpha - {1 \over 2}}.
 $$
Finally, to obtain \eqref{D2omega}, we come back to \eqref{equationvort}, and write 
\begin{align*}
- \tilde \nu \Delta \omega & = f_\omega \: \text{ in } \tilde \Omega^\eps, \\
\omega\vert_{\pa \tilde \Omega^\eps} & = 2 D_z (\tilde u^\app) n \cdot \tau + \varepsilon \lambda \tilde u^\app \cdot \tau + (2 \kappa+\varepsilon\lambda) v\cdot \tau
\end{align*}
where 
\begin{align*}
f_\omega & = - \partial_{\hat\tau}\omega - (\tilde u^\mathrm{app} + v)\cdot \nabla_z \omega - v\cdot \nabla_z \omega^\mathrm{app} + \e \curl_z \tilde R^{\nu, \varepsilon}_\mathrm{app} \\
 & = \curl_z \left( - \partial_{\hat \tau} v + \eps \tilde R^{\nu, \varepsilon}_\mathrm{app} \right) - {\rm div}_z \left( v \, \omega^\mathrm{app} + (\tilde u^\mathrm{app} + v) \omega \right).
 \end{align*}
Using the second writing, the estimate \eqref{dtvLinftyL2}, the bound $\| \eps \tilde R^{\nu, \varepsilon}_\mathrm{app} \|_{L^\infty L^2} \lesssim\eps^{\alpha+\frac{3}{2}}$, the inequality $\| v \|_{L^\infty L^2} \lesssim \eps^{\alpha-\frac{1}{2}}$ given in \eqref{vLinftyL2}, the fact that $\|v\|_{L^\infty L^\infty} \lesssim \eps^\alpha$, and the estimate \eqref{dtvL2L2} (which yields a $L^\infty L^2$ bound on $\omega$), we obtain: 
$$ \| f_\omega \|_{L^\infty(0,T/\eps ; H^{-1}(\tilde \Omega^\eps))} \lesssim \tilde \nu^{-1} \eps^{\alpha-\frac12}. $$
Moreover, by the inequalities $\|\kappa\|_{L^\infty} + \varepsilon \|\lambda \|_{L^\infty}\lesssim \eps^\alpha$, 
$ \| v \|_{L^\infty H^1} \lesssim \tilde \nu^{-1}\eps^{\alpha-\frac{1}{2}}$ 
({\it cf} \eqref{vLinftyL2}-\eqref{dtvLinftyL2}), and by the inequality $\| 2 D_z (\tilde u^\app) n \cdot \tau + \varepsilon \lambda \tilde u^\app \cdot \tau \|_{L^\infty H^1} \lesssim \eps^{\alpha}$, we have 
$$ \| 2 D_z (\tilde u^\app) n \cdot \tau + \varepsilon \lambda \tilde u^\app \cdot \tau + (2 \kappa+\varepsilon\lambda) v\cdot \tau \|_{L^\infty(0,T/\eps ; H^{1/2}(\pa \tilde \Omega^\eps))} \lesssim \tilde \nu^{-1} \eps^{2 \alpha - \frac{1}{2}}. $$
We deduce: 
\begin{equation*} 
\| \omega \|_{L^\infty(0,T/\eps ; H^1(\tilde \Omega^\eps))} \: \lesssim \: \tilde \nu^{-2} \eps^{\alpha-\frac12}. 
\end{equation*}
From the first inequality of Proposition~\ref{prop-omegav} (with $k=2$), we deduce in turn: 
\begin{equation*} % \label{vLinftyH2}
\| v \|_{L^\infty(0,T/\eps ; H^2(\tilde \Omega^\eps))} \: \lesssim \: \tilde \nu^{-2} \eps^{\alpha-\frac12}. 
\end{equation*}

Using \eqref{dtvLinftyH1} and this last bound, it is then easy to see that 
$$ \| f_{\omega} \|_{L^\infty(0,T/\eps; L^2(\tilde \Omega^\eps))} \lesssim \tilde \nu^{-2} \eps^{\alpha-\frac12}$$
while 
$$ \| 2 D_z (\tilde u^\app) n \cdot \tau + \varepsilon \lambda \tilde u^\app \cdot \tau + (2 \kappa+\varepsilon\lambda) v\cdot \tau \|_{L^\infty(0,T/\eps ; H^{3/2}(\pa \tilde \Omega^\eps))} \lesssim \tilde \nu^{-2} \eps^{2 \alpha - \frac{1}{2}} .$$
The inequality \eqref{D2omega} follows. 

\subsection{Proof of Theorem~\ref{theo-main}}

We have now all the estimates necessary to our bootstrap argument.
Assume that $\tilde \nu \lesssim \varepsilon^6$ and fix a $K\geq 3$. We define by continuity
\[
T_{\varepsilon}:= \sup \Big\{ T\in (0,T_{0}], \ \sup_{0\le \hat \tau\le T/\varepsilon}\|\nabla_z v (\hat \tau)\|_{L^\infty(\tilde \Omega^\epsilon)}
+ \sup_{0\le \hat \tau\le T/\varepsilon} \tilde \nu^{K } \| v (\hat \tau)\|_{H^2(\tilde \Omega^\epsilon)} \leq 1 \Big\}.
\]
For $\varepsilon$ small enough, the elliptic Proposition~\ref{prop-omegav} together with Propositions~\ref{prop-vort} and~\ref{prop-higher} gives
 $$ \|v\|_{L^\infty(0,T_{\varepsilon}/\varepsilon;H^3)} \lesssim {\tilde \nu}^{- { 5 \over 2}} \e^{\alpha - {1 \over 2}}+\varepsilon^{\alpha-\frac12}\lesssim {\tilde \nu}^{- { 5 \over 2}} \e^{\alpha - {1 \over 2}}.$$
Next, the second estimate of the elliptic Proposition~\ref{prop-omegav}, together with Proposition~\ref{prop-infty} and the previous estimate, yields
 $$ \| \nabla v \|_{L^\infty(0,T_{\varepsilon}/\varepsilon;L^\infty)} \lesssim \varepsilon^2 + (\varepsilon^\alpha +\varepsilon^{2\alpha} ) \ln (2 + \varepsilon^{-3}{\tilde \nu}^{- { 5 \over 2}} \e^{\alpha - {1 \over 2}}).$$
As $\tilde \nu^{- { 5 \over 2}}\leq \varepsilon^{- { 5N1 \over 2} }$, there exists $\varepsilon_{0}$ such that for all $\varepsilon\leq \varepsilon_{0}$ we have:
\[
 \| \nabla v \|_{L^\infty(0,T_{\varepsilon}/\varepsilon;L^\infty)} + \tilde \nu^K \|v\|_{L^\infty(0,T_{\varepsilon}/\varepsilon;H^2)} \leq \frac{3}4.
\]
Due to the definition of $T_{\varepsilon}$ we deduce that $T_{\varepsilon}=T_{0}$ and that all the estimates stated in Propositions~\ref{prop-vort}, \ref{prop-infty} and \ref{prop-higher} hold for $T=T_{0}$.

Coming back to the original variables yields Theorem~\ref{theo-main}. 
 \bigskip

\noindent
{\bf Acknowledgements.} D.G.V., C.L. and F.R. are partially supported by the Agence Nationale de la Recherche, Project DYFICOLTI, grant ANR-13-BS01-0003-01. 
C.L. is also partially supported by the Agence Nationale de la Recherche, Project IFSMACS, grant ANR-15-CE40-0010. TN's research was supported in part by the NSF under grant DMS-1405728.

\appendix
 
 \section{Traces Lemmas}
 In this section, which verifies that the embedding of $H^1(\Omega^\varepsilon)$ in $
L^2(\pa \Omega^\eps)$ does not depend on $\varepsilon$.
\begin{lemma}\label{lem-trace}
There exists $C$ independent of $\varepsilon$ such that 
 \[
 \| f \|_{L^2(\pa \Omega^\eps)} \leq C \| f \|_{L^2(\Omega^\varepsilon)}^{1/2}\| \partial_{x_{2}} f \|_{L^2(\Omega^\varepsilon)}^{1/2}
 \]
 for any $f\in H^1(\Omega^\varepsilon)$.
\end{lemma}
\begin{proof}
By a density argument, it is enough to prove the inequality for $f\in C^\infty_{c}(\overline{\Omega^\varepsilon})$.

For any $x_{1}\in \mathbb{T}$, we have
\[
 f^2(x_{1}, \varepsilon^{1+\alpha} \eta(x_{1}/\varepsilon)) 
 = 2\int_{ \varepsilon^{1+\alpha} \eta(x_{1}/\varepsilon)}^\infty f(x_{1},x_{2})\partial_{x_{2}} f(x_{1},x_{2}) \, dx_{2}
\]
hence
\begin{align*}
 \| f \|_{L^2(\pa \Omega^\eps)}^2 = &\int_{0}^1 f^2(x_{1}, \varepsilon^{1+\alpha} \eta(x_{1}/\varepsilon)) \langle \varepsilon^\alpha \eta' (x_{1}/\varepsilon)\rangle \, dx_{1}\\
 \leq & - \sqrt{1 + \varepsilon^{2\alpha} \| \eta' \|_{L^\infty}^2} 2 \int_{\Omega^\varepsilon} f\partial_{x_{2}} f \, dx
 \end{align*}
 which gives the result with $C=\sqrt{2} (1 + \| \eta' \|_{L^\infty}^2)^{1/4}$.
\end{proof}

By a scaling argument, $g(z)=f(\varepsilon z)$ we get the similar version on $\tilde \Omega^\varepsilon$:
\begin{equation}\label{eq-trace}
 \| g\|_{L^2(\pa \tilde \Omega^\eps)} \leq C \| g \|_{L^2(\tilde \Omega^\varepsilon)}^{1/2}\| \partial_{z_{2}} g \|_{L^2(\tilde \Omega^\varepsilon)}^{1/2}
\end{equation}
 for any $g\in H^1(\tilde\Omega^\varepsilon)$.
 
 In a similar way, we prove the following trace lemma which involves only the $\curl$ for divergence free vector fields.
 
\begin{lemma}\label{lem-trace2}
 There exists $C$ independent of $\varepsilon$ such that 
 \[
 \| v \|_{L^2(\pa \tilde \Omega^\eps)} \leq C \| v \|_{L^2(\tilde \Omega^\varepsilon)}^{1/2} \| \curl v \|_{L^2(\tilde \Omega^\varepsilon)}^{1/2} 
 \]
 for any $v \in H^1(\tilde\Omega^\varepsilon)$ such that $\Div v=0$ and $v\cdot n\vert_{\pa \tilde \Omega^\eps}=0$.
\end{lemma}
\begin{proof}
As in the previous lemma, we only perform the proof for smooth $v$.

From the proof of the previous lemma, we get
\begin{align*}
 \| v \|_{L^2(\pa \tilde \Omega^\eps)}^2=\| v_{1} \|_{L^2(\pa \tilde \Omega^\eps)}^2 + \| v_{2} \|_{L^2(\pa \tilde \Omega^\eps)}^2 
 &\leq - \sqrt{1 + \varepsilon^{2\alpha} \| \eta' \|_{L^\infty}^2} 2 \int_{\tilde \Omega^\varepsilon} (v_{1}\partial_{z_{2}} v_{1} + v_{2}\partial_{z_{2}} v_{2})\, dz\\
 &\leq - \sqrt{1 + \varepsilon^{2\alpha} \| \eta' \|_{L^\infty}^2} 2 
 \int_{\tilde \Omega^\varepsilon} (-v_{1} \curl v + v \cdot \nabla v_{2})\, dz\\
 &\leq 2 \sqrt{1 + \varepsilon^{2\alpha} \| \eta' \|_{L^\infty}^2} \| v \|_{L^2} \| \curl v \|_{L^2}
\end{align*}
where we have used $\Div v=0$ and $v\cdot n=0$. This ends the proof.
\end{proof}

A corollary of the previous lemma is the following.
\begin{lemma}\label{lem-trace3}
 There exists $C$ independent of $\varepsilon$ such that 
 \[
 \| \nabla v \|_{L^2(\tilde\Omega^\varepsilon)} \leq \| \curl v \|_{L^2(\tilde \Omega^\varepsilon)} + C\varepsilon^{\alpha}\| v \|_{L^2(\tilde \Omega^\varepsilon)}
 \]
 for any $v \in H^1(\tilde\Omega^\varepsilon)$ such that $\Div v=0$ and $v\cdot n\vert_{\pa \tilde \Omega^\eps}=0$.
\end{lemma}
\begin{proof}
We compute
\[
\| \nabla v \|_{L^2(\tilde\Omega^\varepsilon)}^2 = \int_{\tilde\Omega^\varepsilon} (\curl v)^2 + (\Div v)^2 + 2\int_{\tilde\Omega^\varepsilon} \nabla v_{1}\cdot \nabla^\perp v_{2}= \| \curl v \|_{L^2(\tilde \Omega^\varepsilon)}^2 -2\int_{\pa \tilde \Omega^\eps} v_{1} \nabla v_{2} \cdot \tau.
\]
From the expression of $n$ in terms of $\eta$ \eqref{def-ntau}, the condition $v\cdot n$ reads $-\varepsilon^\alpha \eta' v_{1} +v_{2}=0$ hence 
\[
\langle \varepsilon^\alpha \eta'\rangle \partial_{\tau} v_{2}= \partial_{x_{1}} [v_{2} (x_{1},\varepsilon^\alpha \eta (x_{1})) ] = \varepsilon^\alpha \eta'' v_{1}+\varepsilon^\alpha \eta' \partial_{x_{1}} [v_{1} (x_{1},\varepsilon^\alpha \eta (x_{1})) ]
\]
hence
\[
\int_{\pa \tilde \Omega^\eps} v_{1} \nabla v_{2} \cdot \tau = \int_{\mathbb{T}_{\frac1\varepsilon}} v_{1} \partial_{\tau} v_{2} \langle \varepsilon^\alpha \eta'\rangle 
 =\varepsilon^\alpha \int_{\mathbb{T}_{\frac1\varepsilon}} \eta'' v_{1}^2 + \frac12 \eta' \partial_{x_{1}}v_{1}^2 = \frac{\varepsilon^\alpha }2 \int_{\mathbb{T}_{\frac1\varepsilon}} \eta'' v_{1}^2 =\frac{\varepsilon^\alpha }2\int_{\pa \tilde \Omega^\eps} \frac{\eta'' v_{1}^2}{ \langle \varepsilon^\alpha \eta'\rangle}.
\]
Then we use Lemma~\ref{lem-trace2} to conclude
\begin{align*}
 \| \nabla v \|_{L^2(\tilde\Omega^\varepsilon)}^2& \leq \| \curl v \|_{L^2(\tilde \Omega^\varepsilon)}^2 + \varepsilon^\alpha \|\eta''\|_{L^\infty} \| v_{1} \|_{L^2(\pa \tilde \Omega^\eps)}^2
 \\&\leq \| \curl v \|_{L^2(\tilde \Omega^\varepsilon)}^2 + C \varepsilon^\alpha \| v \|_{L^2(\tilde \Omega^\varepsilon)} \| \curl v \|_{L^2(\tilde \Omega^\varepsilon)}\\
 &\leq (\| \curl v \|_{L^2(\tilde \Omega^\varepsilon)} + \tilde C\varepsilon^\alpha \| v \|_{L^2(\tilde \Omega^\varepsilon)}) ^2.
\end{align*}
\end{proof}

 \section{Elliptic estimates on $\mathbb{T} \times \mathbb{R}_{+}$}
 In this section we consider the equation
 \begin{equation}
 \label{elliptic1} \Delta \Psi = F, \quad (x,z) \in \mathbb{T} \times \mathbb{R}_{+},
 \end{equation}
 with the homogeneous Dirichlet boundary condition 
 \begin{equation}
 \label{elliptic2}\Psi(x,0)= 0.
 \end{equation}
 We always assume that $\int_{\mathbb{T}} F\, dx = 0$. We can consider 
 the solution $\Psi$ of \eqref{eqPsi} given by
 \begin{equation}
 \label{solPsi} \Psi(x,z)= \int_{0}^{+ \infty} \sum_{k \neq 0} G_{k}(z,y) F_{k}(y) e^{ik \cdot x}\, dy, \quad
 x \in \mathbb{T}, \, z\geq 0
 \end{equation}
 where the Green function of the Dirichlet problem is given by 
 \begin{equation}
\label{greendirichlet}G_{k}(z,y)= - \left\{\begin{array}{ll} {1 \over |k|} e^{-|k| z} \sinh (|k| y), z>y, \\
 {1 \over |k|} e^{-|k| y} \sinh (|k| z), z<y.
 \end{array} \right.\end{equation}
The goal of this section is to prove the following estimates on the solution of the Laplace problem.
 \begin{proposition}
 \label{propEl}
 We have the estimates:
 \begin{align}
 \label{El1}
 & \forall s \geq 1, \quad
 \| \nabla \Psi \|_{H^s(\mathbb{T} \times \mathbb{R}_{+})} \lesssim \| F \|_{H^{s-1}(\mathbb{T} \times \mathbb{R}_{+})}; \\
 \label{El2} & \|D^2 \Psi \|_{L^\infty(\mathbb{T}\times \mathbb{R}_{+})} \lesssim 1 +
 \| F \|_{L^\infty(\mathbb{T}\times \mathbb{R}_{+})}\ln\left( 2 + \|F \|_{H^2(\mathbb{T}\times \mathbb{R}_{+})}
 \right); \\
 \label{El3} & \|\nabla \Psi \|_{L^\infty(\mathbb{T} \times \mathbb{R}_{+})} \lesssim 1 + \| F \|_{L^2(\mathbb{T} \times \mathbb{R}_{+})}
 \left( \ln \left( 2 + \|F \|_{H^1(\mathbb{T}\times \mathbb{R}_{+})}\right) \right)^{1 \over 2}; \\
 \label{El4} &
 \|\nabla \Psi \|_{L^\infty(\mathbb{T} \times \mathbb{R}_{+})} \lesssim 1 + \| H \|_{L^\infty(\mathbb{T} \times \mathbb{R}_{+})}
 \ln \left( 2 + \|H\|_{H^2(\mathbb{T}\times \mathbb{R}_{+})}\right), \quad \mbox{ if } F = \Div H. 
 \end{align}
 \end{proposition} 

 Let us first prove \eqref{El1} for $s=1$.
 From \eqref{greendirichlet}, we have that 
 $$ |G_{k}(z, y)| \lesssim {1 \over |k|} e^{- 2c_{0}|k| |z-y|}, \quad |\partial_{z}G_{k}(z, y)| \lesssim e^{- 2c_{0}|k| |z-y|}$$
 for some $c_{0}>0.$
By using Parseval and Young, we easily get that
 $$ \|D^2 \Psi \|_{L^2(\mathbb{T} \times \mathbb{R}_{+})}^2 \lesssim
 \sum_{k}\int_{0}^{+ \infty} \Big| \int_{0}^{+\infty} |k| e^{- c_{0}|k||z- y|} |F_{k}(y)| \, dy \Big|^2 \,dz
 \lesssim \sum_{k \neq 0} \|F_{k}\|_{L^2(\mathbb{R}_{+})}^2 \lesssim \|F\|_{L^2(\mathbb{T}\times
 \mathbb{R}_{+})}^2.$$
 In particular, we have obtained that
 $$ \| \nabla \Psi \|_{L^2(\mathbb{T} \times \mathbb{R})} + \|D^2 \Psi \|_{L^2(\mathbb{T} \times \mathbb{R})} \lesssim \|F\|_{L^2(\mathbb{T} \times \mathbb{R})}.$$
 Higher order estimates follow exactly in the same way using $|k|^{\beta_{1}+1} |\partial_{z}^{\beta_{2}} G_{k}| \lesssim |k|^{\beta_{1}+\beta_{2}}e^{- 2c_{0}|k| |z-y|}$ with $\beta_{1}+\beta_{2} \leq s$.
 
 Let us prove \eqref{El2}. 
 We will use a Littlewood Paley partition of unity on $\mathbb{R}$:
 $$ 1 = \chi_{0}(\xi) + \sum_{n \geq 1} \chi_{n}(\xi)$$
 with $ \chi_{n}(\xi)= \chi(\xi/2^n)$
 and $\chi_{0}$ supported in a ball of radius $<1$, $\chi$ supported in an annulus.
 
 For a function $m(\xi)$ defined on $\mathbb{R}$, we define a Fourier multiplier
 $m(D_{x})$ on $\mathbb{T}$ by
 $$ m(D_{x}) f (x)= \sum_{k \in \mathbb{Z}} m(k) f_{k}(\xi) e^{ik x}.$$
 The crucial continuity lemma that we will use is the following.
 \begin{lem}
 \label{lemmult}
 Assume that $\mathcal{F}^{-1}m (x): = \int_{\mathbb{R}} e^{ix \cdot \xi} m(\xi)\, d\xi \in L^1(\mathbb{R})$.
 Then for every $p \in [1, + \infty]$, $m(D)$ is a bounded operator on 
 $L^p(\mathbb{T})$ and 
 $$ \| m(D)\|_{\mathcal{L}(L^p(\mathbb{T} \times \mathbb{R}))} \leq \|\mathcal{F}^{-1}m\|_{L^1(\mathbb{R})}.$$
 \end{lem}
 
 By using the Green function \eqref{greendirichlet}, we can use again the representation
 \eqref{solPsi}.
 Note that the $k=0$ frequency is not present since $F_{0}= 0$ and that we can replace
 in the sum $G_{k}$ by the function
 $G(k,z,y)= G_{k}(z,y) ( 1 - \chi_{0}(100 k))$
 which is defined on $ \mathbb{R} \times \mathbb{R}_{+}^2$ and is smooth with respect to $k$.
 We can thus write for $\alpha \in \mathbb{N}^2$ with $|\alpha|= 2$, $\alpha \neq (0, 2)$,
 $$ \partial^\alpha \Psi(x,z) =\partial_{x}^{\alpha_{1}}\partial_{z}^{\alpha_{2}} \Psi(x,z) = \int_{0}^{+ \infty} \sum_{k} G^\alpha (k,z,y) F_{k}(y) e^{ik \cdot x}\, dy, \quad
 x \in \mathbb{T}, \, z\geq 0$$ 
 where $G^\alpha$ satisfies the estimate
 \begin{equation}
 \label{GCZ} |G^\alpha(k,z,y)| \lesssim |k| e^{- c_{0} |k| |z-y|}.
 \end{equation}
 By using the definition of Fourier multipliers on the torus, we can rewrite this expression as
 $$ \partial^\alpha \Psi(x,z)= \int_{0}^{+ \infty} G^\alpha(D_{x}, z, y) F(\cdot, y) \, dy.$$
 With the help of the Littlewood-Paley decomposition, we write
 $$ \partial^\alpha \Psi(x,z)= \int_{0}^{+ \infty} \Big( \sum_{0<n \leq N} G^\alpha(D_{x}, z, y)
 \chi_{n}(D_{x} ) F(\cdot, y) + \sum_{n >N} G^\alpha(D_{x}, z, y)
 \chi_{n}(D_{x}) F(\cdot, y) \Big) \, dy$$
 For every $z$ and $y$, let us study the inverse Fourier transform of the kernel
 $G^\alpha(\xi, z, y)
 \chi_{n}(\xi )$ which is given by
 $$ \mathcal{F}^{-1}_{\xi}( G^\alpha \chi_{n}) (x,z,y)= \int_{\mathbb{R}} e^{ix \cdot \xi} G^\alpha (\xi, z, y) 
 \chi({\xi \over 2^n}) \, d\xi.$$ 
 By using the pointwise estimate on $G^\alpha$, we have
 $$ | \mathcal{F}^{-1}_{\xi}( G^\alpha \chi_{n}) (x,z,y)| \lesssim 2^{n}e^{-2^{n-2} |z-y|} 
 2^{n}$$
 and by integration by parts, we also have that for every $m$
 $$ |x^m| | \mathcal{F}^{-1}_{\xi}( G^\alpha \chi_{n}) (x,z,y) | \lesssim 2^{n}e^{-2^{n-2} |z-y|} 2^{n(1-m)}.$$
 This yields
 $$ | \mathcal{F}^{-1}_{\xi}( G^\alpha \chi_{n}) (x,z,y) | \lesssim 2^{n}e^{-2^{n-2} |z-y|} { 2^n \over 1 + (2^n |x|)^m}$$
 and hence by taking $m=2$, we get that
 $$ \| \mathcal{F}^{-1}_{\xi}( G^\alpha \chi_{n}) (\cdot ,z,y) \|_{L^1(\mathbb{R})}
 \lesssim 2^{n}e^{-2^{n-2} |z-y|}.$$
 From lemma~\ref{lemmult}, we thus get that for every $z$, $y$, 
 $$ \| G^\alpha(D_{x}, z, y)
 \chi_{n}(D_{x} ) F(\cdot, y)\|_{L^\infty_{x}} \lesssim 2^{n}e^{-2^{n-2} |z-y|} \|F(\cdot, y)\|_{L^\infty_{x}}.$$
 This allows to estimate
 \begin{multline*} \Big| \int_{0}^{+ \infty} \Big( \sum_{0<n \leq N} G^\alpha(D_{x}, z, y)
 \chi_{n}(D_{x} ) F(\cdot, y)\Big) \, dy \Big| \\
 \lesssim \int_{0}^{+ \infty} \sum_{0<n \leq N} 2^{n}e^{-2^{n-2} |z-y|} \|F(\cdot, y)\|_{L^\infty_{x}} \, dy
 \lesssim \sum_{0<n \leq N} \|F\|_{L^\infty_{x,y}} \lesssim N \|F\|_{L^\infty(\mathbb{T}
 \times \mathbb{R}_{+})}.
 \end{multline*}
 For the other sum, 
 we write for $\delta >0$ small ($\delta <1/2$) 
 $$ G^\alpha(D_{x}, z, y)
 \chi_{n}(D_{x}) F(\cdot, y)= {G^\alpha(D_{x}, z, y)
 \chi_{n}(D_{x})\over |D_{x}|^\delta} ( |D_{x}|^\delta F(\cdot, y))$$
 so that we obtain 
 $$ \|G^\alpha(D_{x}, z, y)
 \chi_{n}(D_{x}) F(\cdot, y) \|_{L^\infty_{x}} \lesssim 2^{-n \delta} 2^{n}e^{-2^{n-2} |z-y|} \left\| |D|^{\delta}F(\cdot, y)\right\|_{L^\infty_{x}} \lesssim 
 2^{-n \delta} 2^{n}e^{-2^{n-2} |z-y|} \|F(\cdot, y)\|_{H^{1}_{x}}$$
 where the last estimate comes from the one-dimensional Sobolev embedding.
 This yields
 $$
 \Big| \int_{0}^{+ \infty} \Big( \sum_{n > N} G^\alpha(D_{x}, z, y)
 \chi_{n}(D_{x} ) F(\cdot, y)\Big) \, dy \Big| \lesssim \sum_{n>N} 2^{-n \delta} \|F\|_{L^\infty_{y} H^1_{x}}
 \lesssim 2^{-N \delta} \|F\|_{H^2(\mathbb{T}\times \mathbb{R}_{+})}
 $$
 by using again the one-dimensional Sobolev embedding.
 
 We have thus proven that
 $$ \| \partial^\alpha \Psi\|_{L^\infty}
 \lesssim N \| F\|_{L^\infty} + 2^{-N \delta} \|F\|_{H^2}, $$
and hence we obtain the estimate \eqref{El2} by choosing $N$ such that
 $2^{N\delta} = 2 + \|F\|_{H^2}$ for all the second order derivatives except $\partial_{z}^2 \Psi$.
 To get the missing one, it suffices to use directly the equation as usual.
 
 Let us prove \eqref{El3} which is easier. We observe that we can write
 $$(\nabla \Psi)_{k}(z) = \int_{0}^{+ \infty} G^{(1)}_{k}(z,y) F_{k}(y) \, dy, $$
 where the Green function $G^{(1)}$ is bounded by
 $$ |G^{(1)}_{k}(z,y)| \lesssim e^{- c_{0}|k||z-y]}.$$
 Consequently by using Cauchy-Schwarz, we find
 $$ \|(\nabla \Psi)_{k}\|_{L^\infty_{z}} \lesssim { 1 \over |k|^{1 \over 2}} \|F_{k}\|_{L^2_{z}}.$$
 This yields for some $M \geq 1$ to be chosen,
 $$ \| \nabla \Psi \|_{L^\infty}
 \lesssim \sum_{0<|k| \leq M} {1 \over |k|^{1 \over 2}} \|F_{k}\|_{L^2_{z}} + \sum_{|k|>M} { 1 \over |k|^{3 \over 2}}
 \||k| F_{k}\|_{L^{2}_{z}}$$
 and hence by using Cauchy-Schwarz and Parseval, 
 $$ \| \nabla \Psi \|_{L^\infty} \lesssim (\ln( 1 + M))^{1 \over 2} \|F \|_{L^2} + { 1 \over M} \|F \|_{H^1}.$$
 By taking $M= 1 + \| F \|_{H^1}$, we find \eqref{El3}. 
 
 Let us finally prove \eqref{El4}. Since $F= \nabla \cdot H$,
 we can integrate by parts to obtain that for $i= x, \, z$, 
 $$ \partial_{i} \Psi= H_{2} \delta_{i=z} + \int_{0}^{+ \infty} \sum_{k\neq 0} G^{i} (k,z,y) H(y) e^{ik \cdot x}\, dy, \quad
 x \in \mathbb{T}, \, z\geq 0$$
 where the matrix kernel $G^i$ still satisfies the estimate \eqref{GCZ}.
 Consequently, we can proceed as in the proof of \eqref{El2} to obtain that
 $$ \|\partial^{i} \Psi \|_{L^\infty} \lesssim 1 + \|H\|_{L^\infty} \ln ( 2 + \|H\|_{H^2}).$$

 For the sake of completeness, let us finally briefly recall the proof of Lemma~\ref{lemmult}.
 \begin{proof}[Proof of Lemma~\ref{lemmult}]
 We have
 $$ m(D) f (x) = \int_{\mathbb{T}} K(x-y) f(y) \, dy$$
 with the kernel $K$ defined by
 $$ K(x) = \sum_{k} m(k) e^{ik x}.$$
 By the Young inequality on convolutions, it suffices to prove that $K \in L^1(\mathbb{T})$
 to obtain the result. Thanks to the Poisson summation formula, we have that
 $$ K(x)= \sum_{n \in \mathbb{Z}} \mathcal{F}^{-1}m(x+ 2 \pi n)$$
 and hence
 $$ \|K \|_{L^1 (\mathbb{T})} \leq\sum_{n} \| \mathcal{F}^{-1} m(\cdot + 2 \pi n) \|_{L^1(\mathbb{T})}
 \leq \| \mathcal{F}^{-1} m \|_{L^1(\mathbb{R})}.$$
\end{proof}

\section{Elliptic estimates on $\tilde \Omega^\varepsilon$}\label{appdx3}

We derive here two propositions concerning divergence free functions on $\tilde \Omega^\varepsilon$, whose proofs are based on a change of variables in order to use the previous proposition on $\mathbb{T}\times \R_{+}$.

Of course $H^2$ is embedded in $L^\infty$, but the goal of the following proposition is to control the $L^\infty$ norm by the log of the $H^2$ norm and the $L^2$ norm of the vorticity.
\begin{proposition}\label{vinfty-vH2} 
 There exists $C$ independent of $\varepsilon$ such that 
 $$
 \|v\|_{L^\infty(\tilde{\Omega}^{\eps})}
 \leq C\Bigg(\eps+ \eps^{1 \over 2 }\|v\|_{L^2(\tilde{\Omega}^{\eps})} + \| \curl v \|_{L^2(\tilde{\Omega}^{\eps})} 
 \Big( \ln( 2 + { 1 \over \eps^2} \|v\|_{H^2(\tilde{\Omega}^{\eps})}\Big)^{1\over 2}
 + \eps^\alpha \|v\|_{L^\infty(\tilde{\Omega}^{\eps})} \ln( 2 + {1\over \eps^2}
 \|v\|_{H^2(\tilde{\Omega}^{\eps})})\Bigg)$$ 
 for any $v \in H^2(\tilde\Omega^\varepsilon)$ such that $\Div v=0$ and $v\cdot n\vert_{\pa \tilde \Omega^\eps}=0$.
\end{proposition}

\begin{proof} As usual, we set $\omega:= \curl v = \partial_{1}v_{2}-\partial_{2}v_{1}$ and we define the vector field
 \begin{equation}
 \label{defauxu} u(z)= \left( \begin{array}{ll} v_{1}(z_{1} , z_{2}+ \e^\alpha \eta(z_{1})), \\ v_{2}(z_{1}, z_{2} + \e^{\alpha}
 \eta (z_{1})) - \eps^\alpha \eta'(z_{1}) v_{1} (z_{1} , z_{2}+ \e^\alpha \eta(z_{1}))) \end{array} \right)\end{equation}
 which is defined on $\mathbb{T}_{{1 \over \e}} \times \mathbb{R}_{+}.$
 We observe that since $v$ is divergence free with vanishing normal component on the boundary, 
 then we also have that
 $$ \Div u = 0 \text{ in } \mathbb{T}_{{1 \over \e}} \times \mathbb{R}_{+}
\quad \text{and} \quad u_{2}= 0 \text{ on } \mathbb{T}_{{1 \over \e}} \times \{ 0 \}.
 $$
 We can thus introduce a stream function $\phi$ so that
 \begin{equation*}
 u= \nabla^\perp \phi.
 \end{equation*}
 We choose $\phi$ such that
 \begin{equation}
 \label{Phieq}
 \Delta \phi= \omega_{u}:= \partial_{1} u_{2} - \partial_{2} u_{1} \text{ in } \mathbb{T}_{{1 \over \e}} \times \mathbb{R}_{+}
 \end{equation}
 with the Dirichlet boundary condition.
 Note that we get that $\phi_{0}:=\e \int_{z_{1} \in \mathbb{T}_{1\over \e}} \phi$ solves 
\begin{equation}\label{D2phi0}
 \partial_{z_{2}}^2 \phi_{0} = - \e \partial_{2} \int_{z_{1}}u_{1}
\end{equation}
 and thus that
 \begin{equation*}
 \partial_{z_{2}} \phi_{0}= - \e \int_{z_{1}}u_{1}.\end{equation*}
 By setting $\psi = \phi - \phi_{0}$, we get that $ \psi$ solves
 \begin{equation}
 \label{psieq} \Delta \psi= f:= \omega_{u} -\e \int_{z_{1}} \omega_u, \quad \psi_{\vert z_{2}= 0}= 0.
 \end{equation}
 
 Next, we write that
 $$ \|u\|_{L^\infty} \leq \left\|\eps \int_{z_{1}} u_{1} \right\|_{L^\infty_{z_{2}}} +
 \| \nabla \psi \|_{L^\infty}.$$
 By the one-dimensional Sobolev embedding, we obtain that 
\[ \left\|\eps \int_{\mathbb{T}_{1\over \eps}} u_{1} \right\|_{L^\infty_{z_{2}}} \lesssim
 \left\| \eps\int_{\mathbb{T}_{1\over \eps}} u_{1} \right\|_{L^2_{z_{2}}}
 + \left\|\eps \int_{\mathbb{T}_{1\over \eps}} \partial_{2} u_{1} \right\|_{L^2_{z_{2}}}
 \leq \eps^{1 \over 2} \Big(
 \| u_{1}\|_{L^2} + \| \nabla u_{1} \|_{L^2}\Big).
 \]
 Moreover, thanks to \eqref{defauxu}, we have that
 $$ \|u_{1}\|_{L^2} \lesssim \|v_{1}\|_{L^2}, \quad \| \nabla u_{1} \|_{L^2} \lesssim \|\nabla v_{1}\|_{L^2}$$
 which implies
 $$ \left\|\eps \int_{\mathbb{T}_{1\over \eps}} u_{1} \right\|_{L^\infty_{z_{2}}} \lesssim \eps^{1 \over 2}
 \|v\|_{L^2} + \eps^{1 \over 2} \|\nabla v \|_{L^2}.$$
 Hence, Lemma~\ref{lem-trace3} gives
 \begin{equation}
 \label{vinfty1}
 \| v\|_{L^\infty} \lesssim \|u\|_{L^\infty} \lesssim \eps^{1 \over 2 }\|v\|_{L^2} + \eps^{1 \over 2} \| \omega \|_{L^2}
 + \| \nabla \psi \|_{L^\infty} .
 \end{equation}
It remains to estimate $\| \nabla \psi \|_{L^\infty(\mathbb{T}_{1 \over \eps} \times \mathbb{R}_{+})}$
 where $\psi$ solves \eqref{psieq}.
 Thanks to \eqref{defauxu}, we observe that
 $$ \omega_{u}= \omega \circ X + \partial_{2}\left( \eps^\alpha \eta' v_{2} \circ X\right) - \partial_{1}\left( \eps^\alpha \eta' v_{1} \circ X\right)$$
 where $X(z_{1}, z_{2})= (z_{1}, z_{2} + \eps^\alpha \eta'(z_{1})).$
 We thus split $\psi = \psi_{1} + \psi_{2}$ where $\psi_{1}$ solves in $\mathbb{T}_{1\over \eps}\times \mathbb{R}_{+}$
 $$ \Delta \psi_{1}= f := \omega \circ X - \langle \omega \circ X\rangle_{\eps}, \quad \psi_{1}(z_{1}, 0)= 0, $$
 and $\psi_{2}$ solves
 $$ \Delta \psi_{2}= \Div h,\quad \psi_{2}(z_{1}, 0)= 0, $$ 
 where we have set
 $$ h= \left( - \eps^\alpha \eta' v_{1} \circ X, \eps^\alpha \eta' v_{2} \circ X\right) - \langle \left( - \eps^\alpha \eta' v_{1} \circ X, \eps^\alpha \eta' v_{2} \circ X \right) \rangle_{\eps}$$
 and for a function $f$ defined on $\mathbb{T}_{1\over \eps}\times \mathbb{R}_{+}$, we use the notation
 $\langle f \rangle_{\eps}= \eps \int_{\mathbb{T}_{1\over \eps}} f(z_{1}, z_{2})\, dz_{1}.$ 
 To estimate $\psi_{1}$, we set
 $$ \Psi_{1}(x_{1}, x_{2}) = \psi_{1}({x_{1} \over \eps}, {x_{2} \over \eps}), \quad F(x_{1}, x_{2})= {1 \over \eps^2} f({x_{1}
 \over \eps}, {x_{2}\over \eps})$$
 so that $\Psi_{1}$ and $F$ are defined on $\mathbb{T}\times \mathbb{R}_{+}$ and solve
 $$ \Delta \Psi_{1}= F.$$
 By using the estimate \eqref{El3} of Proposition~\ref{propEl}, we get that
 $$ \| \nabla \Psi_{1}\|_{L^\infty(\mathbb{T}\times \mathbb{R}_{+})} \lesssim 1 + \| F\|_{L^2(\mathbb{T} \times \mathbb{R}_{+})}
 \left( \ln( 2 + \| F\|_{H^1(\mathbb{T}\times \mathbb{R}_{+})}) \right)^{1 \over 2}.$$ 
 In this original coordinates, this yields
 \begin{align}
\nonumber \|\nabla \psi_{1} \|_{L^\infty( \mathbb{T}_{1 \over \eps} \times \mathbb{R}_{+})}
 & \lesssim \eps + \| f \|_{L^2( \mathbb{T}_{1 \over \eps} \times \mathbb{R}_{+})}
 \left( \ln( 2 + { 1 \over \eps^2} \|f\|_{H^1{( \mathbb{T}_{1 \over \eps} \times \mathbb{R}_{+})}})\right)^{1\over 2} \\
\label{vinfty2} & \lesssim \eps + \| \omega \|_{L^2( \tilde{\Omega}^{\eps})}
 \left( \ln( 2 + { 1 \over \eps^2} \|v\|_{H^2(\tilde{\Omega}^{\eps})}\right)^{1\over 2} .
 \end{align}
 We shall now estimate $\psi_{2}$. We use the change of variables, 
 $$ \Psi_{2}(x_{1}, x_{2}) = \psi_{2}({x_{1} \over \eps}, {x_{2} \over \eps}), \quad H(x_{1}, x_{2})= {1\over \eps} h({x_{1}\over \eps}, 
 {x_{2}\over \eps}) $$
 so that again $\Psi_{2}$ and $H$ are defined on $\mathbb{T}\times \mathbb{R}_{+}$ and solve
 $$ \Delta \Psi_{2}= \Div H.$$
 By using the estimate \eqref{El4} of Proposition~\ref{propEl}, we obtain that
 $$ \| \nabla \Psi_{2} \|_{L^\infty(\mathbb{T}\times \mathbb{R}_{+})}
 \lesssim 1 + \| H\|_{L^\infty(\mathbb{T}\times \mathbb{R}_{+})} \ln (2 + \|H\|_{H^2(\mathbb{T}\times \mathbb{R}_{+})})$$
 which gives in the original coordinates
 \begin{align}
\nonumber \| \nabla \psi_{2} \|_{L^\infty( \mathbb{T}_{1 \over \eps} \times \mathbb{R}_{+})}
 & \lesssim \eps + \|h\|_{L^\infty( \mathbb{T}_{1 \over \eps} \times \mathbb{R}_{+})}
 \ln(2 + { 1 \over \eps^2}\| h\|_{H^2( \mathbb{T}_{1 \over \eps} \times \mathbb{R}_{+})} )\\
 \label{vinfty3} & \lesssim \eps + \eps^\alpha \| v\|_{L^\infty(\tilde{ \Omega}^{\eps})} \ln( 2 + { 1 \over \eps^2}
 \|v\|_{H^2(\tilde{\Omega}^{\eps})}).
 \end{align}
 Consequently, by combining \eqref{vinfty1}, \eqref{vinfty2}, \eqref{vinfty3}, we obtain that
 $$ \|v\|_{L^\infty(\tilde{\Omega}^{\eps})}
 \lesssim \eps +\eps^{1 \over 2 }\|v\|_{L^2} + \| \omega \|_{L^2(\tilde{\Omega}^{\eps})} 
 \left( \ln( 2 + { 1 \over \eps^2} \|v\|_{H^2(\tilde{\Omega}^{\eps})}\right)^{1\over 2}
 + \eps^\alpha \|v\|_{L^\infty(\tilde{\Omega}^{\eps})} \ln( 2 + {1\over \eps^2}
 \|v\|_{H^2(\tilde{\Omega}^{\eps})}),$$ 
 which ends the proof.
\end{proof}

In the second proposition, we finally control the quantities appearing in the assumption \eqref{major_hyp}.
\begin{proposition}
\label{prop-omegav}
There exist $C$ and $\varepsilon_{0}$ such that 
 \begin{gather*}
\text{ for all $k \le 3$,} \: \| v \|_{H^k(\tilde\Omega^{\e})}
 \leq C\Big( \|\curl v\|_{H^{k-1}(\tilde\Omega^{\e})} + \|v\|_{L^2(\tilde \Omega^{\e})}\Big) , \\
 \| \nabla v \|_{L^\infty (\tilde \Omega^\e)} \leq C\Big( \e^2 + (\| \curl v \|_{L^\infty(\tilde \Omega^{\e})} + \e^\alpha\| v \|_{W^{1,\infty}(\tilde \Omega^{\e})}) \ln\Big (2 + \e^{-3}\| v\|_{H^3(\tilde\Omega^{\e})} \Big) \Big), 
\end{gather*}
for all $\varepsilon\leq\varepsilon_{0}$ and for any $v \in H^2(\tilde\Omega^\varepsilon)$ such that $\Div v=0$ and $v\cdot n\vert_{\pa \tilde \Omega^\eps}=0$.
\end{proposition}

 \begin{proof} We use the same notations as in the proof of the previous proposition. 
 To use the estimates in $\mathbb{T}\times \R_{+}$ , we change again the variables 
 \[
 \Psi(x_{1}, x_{2})= \psi({x_{1} \over \e},{ x_{2} \over \e})
\quad\text{and}\quad F (x_{1}, x_{2})= {1\over \e^2}f({x_{1} \over \e},{ x_{2} \over \e}),
\]
so that $\Psi$ solves
 \begin{equation} \label{eqPsi}
\left\{ \begin{array}{ll}
 \Delta \Psi= F &\text{in }\mathbb{T}\times \R_{+},\\
 \Psi= 0 &\text{on }\mathbb{T}\times \{0\},
 \end{array}\right.
 \end{equation}
with $F$ such that $\int_{x_{1}} F= 0$.
 In particular, for $k \le 3,$ \eqref{El1} gives
\[
 \|D^{k} \nabla \Psi \|_{L^2(\mathbb{T} \times \mathbb{R})} \leq \|\nabla \Psi \|_{H^{k}(\mathbb{T} \times \mathbb{R})} \lesssim \|F\|_{H^{k-1}(\mathbb{T} \times \mathbb{R})}
\]
 and thus 
 $$ 
 \|D^{k}\nabla \psi \|_{L^2(\mathbb{T}_{1 \over \e} \times \mathbb{R}_{+})}=\varepsilon^k \|D^{k}\nabla \Psi \|_{L^2(\mathbb{T} \times \mathbb{R}_{+})} \lesssim \varepsilon^k \|F\|_{H^{k-1}(\mathbb{T} \times \mathbb{R})} \leq \|f\|_{H^{k-1}(\mathbb{T}_{1 \over \e} \times \mathbb{R}_{+})}.
 $$
 Hence, for $k \le 2$
 \[
 \|D \nabla \psi \|_{H^k(\mathbb{T}_{1 \over \e} \times \mathbb{R}_{+})} \lesssim \| \omega_{u}\|_{H^{k}(\mathbb{T}_{1 \over \e} \times \mathbb{R}_{+})} .
 \]
 Coming back to $u=\nabla^{\perp}\psi+ \nabla^{\perp}\phi_{0}$ and thanks to \eqref{D2phi0}, we write
\begin{align*}
 \| D u \|_{H^k(\mathbb{T}_{1 \over \e} \times \mathbb{R}_{+})} &\lesssim \|D \nabla \psi \|_{H^k(\mathbb{T}_{1 \over \e} \times \mathbb{R}_{+})}
 +\|D^2_{z_{2}} \phi_{0} \|_{H^k(\mathbb{T}_{1 \over \e} \times \mathbb{R}_{+})}\\
 &\lesssim \|\omega_{u}\|_{H^k(\mathbb{T}_{1 \over \e} \times \mathbb{R}_{+})}
 + \varepsilon^{1/2}\| \int_{z_{1}\in \mathbb{T}_{1 \over \e}} \omega_{u} \|_{H^k_{z_{2}}(\mathbb{R}_{+})} \\
 &\lesssim \|\omega_{u}\|_{H^k(\mathbb{T}_{1 \over \e} \times \mathbb{R}_{+})}\lesssim \| \omega \|_{H^k(\tilde \Omega^{\e})} + \e^{\alpha} \|v\|_{H^{k+1}(\tilde \Omega^{\e})},
 \end{align*}
 where we have used the definition of $u$ and $\omega_{u}$ \eqref{defauxu}--\eqref{Phieq}.
 Next, we observe again from the definition of $u$ \eqref{defauxu} that for all $k \le 2$, 
 $$ \|D v \|_{H^k(\tilde \Omega^{\e})} \lesssim \| D u \|_{H^k(\mathbb{T}_{1 \over \e} \times \mathbb{R}_{+})} + \e^\alpha \| v \|_{H^{k+1}(\tilde \Omega^{\e})} $$
 which allows us to deduce the existence of $C>0$ independent of $v$ and $\varepsilon$ such that
 \[
 \| v \|_{H^{k+1}(\tilde \Omega^{\e})} \leq 2C \Big(\| v \|_{L^2(\tilde \Omega^{\e})} + \| \omega \|_{H^k(\tilde \Omega^{\e})} + \e^{\alpha} \|v\|_{H^{k+1}(\tilde \Omega^{\e})}\Big).
 \]
 Setting $\varepsilon_{0}$ such that $2C\varepsilon_{0}^\alpha \leq 1/2$ gives the first estimate of Proposition~\ref{prop-omegav}.
 
To prove the second estimate, we use the estimate \eqref{El2} of Proposition~\ref{propEl}.
 In the original coordinates, this gives
 $$ \|D^2 \psi\|_{L^\infty(\mathbb{T}_{{1 \over \e}}\times \mathbb{R}_{+})}
 \lesssim \e^2 + \|f\|_{L^\infty(\mathbb{T}_{{1 \over \e}}\times \mathbb{R}_{+})}
 \ln(2 + { 1 \over \eps^3} \|f\|_{H^2({\mathbb{T} \over \e}\times \mathbb{R}_{+})}
 ).$$
 Thanks to \eqref{D2phi0}, we also have
 that 
 $$ \|D^2_{z_{2}} \phi_{0}\|_{L^\infty(\mathbb{R}_{+})} = \Big\| \e \int_{z_{1}\in \mathbb{T}_{1\over \e}} \omega_u \Big\|
 _{L^\infty_{z_{2}}} \leq \|\omega_{u}\|_{L^\infty(\mathbb{T}_{{1 \over \e}}\times \mathbb{R}_{+})}.$$ 
 Therefore, we actually obtain that
 $$ \|\nabla u \|_{L^\infty(\mathbb{T}_{{1 \over \e}}\times \mathbb{R}_{+})}
 \lesssim \e^2 + \|\omega_{u}\|_{L^\infty(\mathbb{T}_{{1 \over \e}}\times \mathbb{R}_{+})}
 \ln (2 +{ 1 \over \eps^3}\|\omega_{u}\|_{H^2({\mathbb{T}_{1 \over \e}\times \mathbb{R}_{+})} } )
 .$$
 To conclude, we use that
 \begin{align*}
& \|\nabla v \|_{L^\infty(\tilde \Omega^{\e})} \lesssim \|\nabla u \|_{L^\infty(\mathbb{T}_{{1 \over \e}}\times \mathbb{R}_{+})} + \e^\alpha \|v\|_{L^\infty(\tilde \Omega^{\e})} \\
 & \|\omega_{u}\|_{L^\infty(\mathbb{T}_{{1 \over \e}}\times \mathbb{R}_{+})}\lesssim \|\omega \|_{L^\infty(\tilde \Omega^{\e})} + \e^{\alpha} \|\nabla v \|_{L^\infty(\tilde \Omega^{\e})}+ \e^{\alpha} \| v \|_{L^\infty(\tilde \Omega^{\e})} \\
 & \|\omega_{u}\|_{H^2(\mathbb{T}_{{1 \over \e}}\times \mathbb{R}_{+})}\lesssim \|v\|_{H^3(\tilde \Omega^{\e})},
 \end{align*}
 to deduce the existence of $C$ independent of $v$ and $\varepsilon$ such that
 \[
 \|\nabla v \|_{L^\infty(\tilde \Omega^{\e})} \leq C \Big( \e^2 +(\|\omega \|_{L^\infty(\tilde \Omega^{\e})} + \e^{\alpha} \| v \|_{W^{1,\infty}(\tilde \Omega^{\e})}) \ln (2 +{ 1 \over \eps^3}\|v\|_{H^3(\tilde \Omega^{\e})} ) \Big),
 \]
 which ends the proof.
 \end{proof}
 
\def\cprime{$'$}

\end{document}